\documentclass[11pt,reqno]{amsart}
\usepackage{amsmath,amssymb,geometry,color,tikz-cd}
\usepackage[colorlinks=true, linkcolor=red!80!black, urlcolor=purple, citecolor=blue!70!black]{hyperref}
\usepackage{amssymb}% see geometry.pdf on how to lay out the page. There's lots.
\usepackage[scr=boondox, cal=esstix]{mathalpha}

\usepackage{bbm}
\usepackage{enumitem}
\usepackage{upgreek}
\usepackage{bm}
\usepackage{stmaryrd}
\usepackage{cleveref}
\usepackage[all]{xy}
\allowdisplaybreaks

\newcommand{\A}{\mathbb{A}}

\newcommand{\C}{\mathbb{C}}
\newcommand{\G}{\mathbb{G}}
\newcommand{\Q}{\mathbb{Q}}
\newcommand{\R}{\mathbb{R}}
\newcommand{\Z}{\mathbb{Z}}
\newcommand{\N}{\mathbb{N}}
\renewcommand{\P}{\mathbb{P}}

\newcommand{\cD}{\mathcal{D}}

\newcommand{\cF}{\mathcal{F}}

\newcommand{\cH}{\mathcal{H}}

\newcommand{\cL}{\mathcal{L}}
\newcommand{\cM}{\mathcal{M}}

\newcommand{\cO}{\mathcal{O}}

\newcommand{\cR}{\mathcal{R}}

\newcommand{\cV}{\mathcal{V}}

\newcommand{\cX}{\mathcal{X}}
\newcommand{\cY}{\mathcal{Y}}

\newcommand{\fa}{\mathfrak{a}}

\newcommand{\fm}{\mathfrak{m}}
\newcommand{\fn}{\mathfrak{n}}

\newcommand{\bP}{\mathbb{P}}
\newcommand{\bC}{\mathbb{C}}

\newcommand{\bQ}{\mathbb{Q}}

\newcommand{\bG}{\mathbb{G}}

\newcommand{\bN}{\mathbb{N}}

\newcommand{\PP}{\mathbb{P}}

\newcommand{\isom}{\cong}

\DeclareMathOperator{\Hilb}{Hilb}
\DeclareMathOperator{\lct}{lct}
\DeclareMathOperator{\Spec}{Spec}
\DeclareMathOperator{\vol}{vol}

\DeclareMathOperator{\nvol}{\widehat{vol}}

\DeclareMathOperator{\gr}{gr}

\DeclareMathOperator{\im}{Im}
\DeclareMathOperator{\ord}{ord}

\DeclareMathOperator{\Proj}{Proj}

\DeclareMathOperator{\DivVal}{DivVal}

\newcommand{\PGL}{\mathrm{PGL}}

\newcommand{\Hom}{\mathrm{Hom}}
\newcommand{\bT}{\mathbb{T}}

\newcommand{\tpsi}{\tilde{\psi}}
\newcommand{\fS}{\mathfrak{S}}

\newcommand{\Fe}{F_* ^e}

\numberwithin{equation}{section}

\newtheorem{prop} {Proposition} [section]
\newtheorem{thm}[prop] {Theorem} 

\newtheorem{lem}[prop] {Lemma}
\newtheorem{cor}[prop]{Corollary}
\newtheorem{prop-def}[prop]{Proposition-Definition}

\newtheorem{conj}[prop]{Conjecture}

\newtheorem{thm-defn}[prop]{Theorem-Definition}

\theoremstyle{definition}
 
\newtheorem{rem}[prop] {Remark} 
\newtheorem{defi}[prop] {Definition}

\newtheorem{que}[prop]{Question}

\newcommand{\FA}{\alpha_{F}}

\newcommand{\s}{\mathscr{s}}

\title{On Positivity of the Limit $F$-signature}
\author{Yuchen Liu}
\address{Department of Mathematics, Northwestern University, Evanston, IL 60208, USA}
\email{yuchenl@northwestern.edu}
\author{Suchitra Pande}
\address{Department of Mathematics, University of Utah, Salt Lake City, UT 84112, USA}
\email{suchitra.pande@utah.edu}
\date{\today}

\begin{document}

\begin{abstract}
    We study a conjecture of Carvajal-Rojas, Schwede and Tucker which states that for a complex KLT singularity $(R, \fm)$, the $F$-signatures of the reductions of $R$ to characteristic $p \gg 0$ remain bounded away from zero as $p \to \infty$. We prove that this conjecture holds for three-dimensional non-weakly exceptional singularities by an inductive argument. We also prove that the conjecture holds for smooth hypersurfaces of very low degree by constructing isotrivial normal toric degenerations. By considering the version of this conjecture for the Frobenius-alpha invariant, our techniques are inspired by K-stability theory and involve using degenerations and birational geometry.
\end{abstract}

\maketitle

\section{Introduction}

The \emph{$F$-signature} is a numerical invariant of singularities of local rings in positive characteristics that measures the \emph{asymptotic} flatness of the \emph{Frobenius} morphism. For an $F$-finite local ring $R$ over $\mathbb{F}_p$ (meaning that the Frobenius map on $\Spec(R)$ is finite), the $F$-signature $\s(R)$ is a real number between $0$ and $1$; it equals $1$ exactly when $R$ is a regular ring \cite{HunekeLeuschkeTwoTheoremsAboutMaximal}, \cite{TuckerFSigExists}. This latter fact is a strengthening of the classical theorem of Kunz that characterizes regular rings via the flatness of the Frobenius map \cite{KunzCharacterizationsOfRegularLocalRings}. Similarly, the $F$-signature is positive exactly when $R$ has only \emph{strongly $F$-regular} singularities \cite{AberbachLeuschke}, a mild class of singularities that implies that $R$ is normal and Cohen-Macaulay \cite{HochsterHunekeFRegularityTestElementsBaseChange}. The numerical measurement encoded by the $F$-signature has found important applications, notably to bounding the local fundamental groups and torsion in divisor class groups of singularities \cite{CarvajalRojasSchwedeTuckerFundamentalGroups}, \cite{PolstraATheoremaboutMCMmodule}, \cite{MartinTorsionindivisorclassgroup}.

Although strongly $F$-regular singularities are defined via splittings of the Frobenius map of $R$, they turn out to be deeply connected to a class of complex singularities arising in the minimal model program, namely, \emph{Kawamata log terminal (KLT)} singularities. Indeed, it follows from the works \cite{HaraWatanabeFRegFPure}, \cite{HaraYoshidaGeneralizationOfTightClosure} and \cite{TakagiInterpretationOfMultiplierIdeals} that a normal, $\Q$-Gorenstein ring $R$ of finite type over $\C$ is KLT if and only if the reduction of $R$ to characteristic $p \gg 0$ is strongly $F$-regular for all $p \gg 0$. Since the $F$-signature is positive for strongly $F$-regular rings, this suggests a connection between the properties of a KLT singularity $R$ over $\C$ and the $F$-signature of its reductions to positive characteristics. In this direction, Carvajal-Rojas, Schwede and Tucker made the following conjecture in \cite{CarvajalRojasSchwedeTuckerFundamentalGroups}:
\begin{conj} \label{mainconjintro}
    Let $(R,\fm)$ be a KLT singularity that is essentially of finite type over $\C$. For any prime $p \gg 0$, let $\s_p (R)$ denote the $F$-signature of the reduction of $R$ to characteristic $p $ (see \Cref{liminfFsigdfn} for a more precise formulation). Then, there is a constant $C >0$ (independent of $p$) such that for all $p \gg 0$, we have $\s_p (R) > C$. In other words, we have
    \[  \liminf_{p \to \infty} \s_p(R) > 0 .\]
\end{conj}

This conjecture has been established in the following cases: toric singularities \cite{SinghFSignatureOfAffineSemigroup}, \cite{VonKorffthesis}, finite quotient singularities \cite{HunekeLeuschkeTwoTheoremsAboutMaximal}, quadric hypersurfaces \cite{TrivediHKDensitfunctionofquadrics}, diagonal hypersurfaces \cite{CaminataShidelerTuckerZermanFsignatureofdiagonalhypersurfaces}, full flag varieties \cite{PandeFrobeniusAlpha} and reductive quotient singularities \cite{TakagiYamaguchiUniformPositivity}. In this paper, we develop several tools to address this conjecture and apply them to prove it in many new cases. Our first observation is that \Cref{mainconjintro} for KLT singularities can be reduced to a similar conjecture about the \emph{Frobenius-alpha invariant} for log Fano pairs (\Cref{alphadefinitions}):

\begin{conj} \label{mainconj2intro}
    Let $(X, \Delta)$ be a projective log Fano pair over $\C$ (\Cref{def:logFano}). There exists a constant $C>0$ such that for all $p \gg 0$ and reductions $(X_p, \Delta_p)$, we have $\FA (X_p, \Delta_p) \geq C$.
\end{conj}
Relating the two conjectures, we prove:
\begin{thm} \label{alphavssintro}(\Cref{thm:localtoglobalreduction}, \Cref{alphavss})
    \Cref{mainconjintro} holds for all KLT pairs $(R,\Delta)$ of dimension $n \geq 2$ if and only if \Cref{mainconj2intro} holds for log Fano pairs in dimension $n-1$.
\end{thm}

This result was also independently obtained by Takagi and Yamaguchi in \cite{TakagiYamaguchiUniformPositivity}. The Frobenius-alpha invariant of a log Fano pair in positive characteristics was introduced in \cite{PandeFrobeniusAlpha} as an analog of Tian's alpha-invariant from complex geometry. Thus, it is more amenable to tools involving degenerations and birational geometry, which we exploit. We also note that \Cref{mainconj2intro} for a $\Q$-Fano variety $X$ is equivalent to \Cref{mainconjintro} for the \emph{cone over $X$} with respect to the embedding defined by any $-rK_X$ (for any $r>0$).

Furthermore, we establish a partial inductive approach to proving \Cref{mainconjintro}. For a projective log Fano pair $(X, \Delta)$, we let $\alpha(X, \Delta)$ denote Tian's alpha-invariant, or equivalently, the global log canonical threshold (\Cref{def:complexalpha}).

\begin{thm}[\Cref{inductivethm}] \label{mainthm3intro}
    Suppose \Cref{mainconj2intro} is true for all complex log Fano pairs of dimension $n \geq 1$. Then, \Cref{mainconj2intro} also holds for all log Fano pairs $(X, \Delta)$ of dimension $n+1$ with $\alpha (X, \Delta) < 1$.
\end{thm}

We now discuss our results in more detail.

\subsection{Threefold singularities.} We apply \Cref{mainthm3intro} to complex threefold KLT singularities that are not \emph{weakly exceptional}. A weakly-exceptional singularity is defined as a KLT singularity $x \in (X, \Delta)$ that admits a unique PLT blow-up; see \cite{ProkhorovBlowupsofcanonicalsingularity}. Weakly exceptional singularities are local versions of log Fano pairs $(S, \Delta_S) $ such that $\alpha (S, \Delta_S) >1$. 
\Cref{mainthm3intro} and \Cref{alphavssintro} together reduce \Cref{mainconjintro} for non-weakly exceptional threefold singularities to the case of pairs $(\P^1, \Delta)$ where $\Delta$ has at most three components (see \Cref{inductivethm}). The latter has been established in \cite{BCPTlimitFsigarxiv} and thus we obtain:
\begin{thm}
   \Cref{mainconjintro} holds for all non-weakly exceptional threefold KLT singularities $(R, \fm)$ over $\C$. Similarly, \Cref{mainconj2intro} holds for all log del Pezzo surfaces $S$ satisfying $\alpha (S) < 1$.
\end{thm}

This theorem covers the case of smooth del Pezzo surfaces of degree at least $2$ and is new for all the non-toric and non-diagonal ones. In forthcoming work, we use results of Monsky \cite{MonskyMasonstheorem} to extend this result to all non-weakly exceptional pairs $x \in (X, \Delta)$ in dimension $3$. We also obtain uniform versions of \Cref{mainconjintro} for non-weakly exceptional threefold singularities by comparing the $F$-signature to the \emph{normalized volume} (\Cref{defi:normalizedvolume}) defined by Li \cite{LiMinimizingNormalizedVolumes}. See Section \ref{section:threefolds} for the details.

% While we are not able to answer this question in dimension $3$, we obtain the following weaker version:
%\begin{thm}[\Cref{thm:effectivelimFsig}]
  %  Let $\varepsilon > 0 $ be any fixed real number and $\cS_{\varepsilon}$ denote the set of three-dimensional KLT singularities $(R, \fm)$ (over $\C$) such that every Koll\'ar component $(S, \Delta_S)$ of $X$ satisfies $ \alpha (S, \Delta_S) < 1 - \varepsilon$. Then, there exists a constant $c : = c( \varepsilon) $ such that for any $R \in \cS_\varepsilon$, we have 
  %  \[ \liminf_{p \to \infty} \s_ p (R_p) \geq c \, \nvol (R). \]
    
%\end{thm}

Our techniques also allow us to prove an interesting result about $F$-regularity (resp. $+$-regularity) of non-weakly exceptional three-dimensional singularities in positive and mixed characteristics when the residue characteristic is larger than $5$. More precisely, we have

\begin{thm} (\Cref{+regforthreedimsingularities}) \label{introthm:Fregularity}
    Let $X = \Spec(R)$ be a three-dimensional KLT singularity over a perfect field $k$ of characteristic $p>5$ (resp. a complete DVR $A$ with residue field $k$). Assume that $X$ is non-weakly exceptional (\Cref{defi:weaklyexceptionalsing}). Then, $X$ is $F$-regular (resp. $+$/BCM-regular).
\end{thm}

By the examples constructed in \cite{CasciniTanakaWitaszekcounterexamplestoFregularity}, it is known that the above theorem fails without the non-weakly exceptional assumption.

\subsection{Toric degenerations.} We also consider another method for proving \Cref{mainconj2intro} which involves using an \emph{isotrivial degeneration} of a Fano variety $X$ to a \emph{normal toric variety} $X_0$ (see \Cref{isotrivialdfn}). An isotrivial degeneration $\pi: \cX \to C$ over a pointed curve $0 \in C$ has the special property that all fibers of $\pi$ over the punctured curve $C \setminus \{0\}$ are isomorphic to $X$. Combined with the semi-continuity properties of the $F$-signature or the $\FA$-invariant, this provides the required control while reducing modulo $p \gg 0$:

\begin{thm}
     Suppose a complex $\Q$-Fano variety $X$ admits an isotrivial degeneration $\pi: \cX \to C$ to a normal toric variety $X_0$ over a pointed curve $0 \in C$. Assume that $-K_{\cX/C}$ is $\pi$-ample. Then, \Cref{mainconj2intro} holds for $X$.
\end{thm}

This theorem readily implies \Cref{mainconj2intro} for spherical $\Q$-Fano varieties over $\C$ (\Cref{cor:spherical}) thanks to the normal toric degenerations constructed by \cite{AlexeevBrionToricDegenerationsofspherical}. This includes Grassmannians and partial flag varieties. Furthermore, we construct new isotrivial degenerations to normal toric varieties for smooth low-degree hypersurfaces: 

\begin{thm}[\Cref{thmlowdegree}] \label{mainthm2intro}
    Let $X \subset \P^n$ be a smooth hypersurface of degree $d$ in $\P^n$ such that $n \geq \binom{2d-1}{d} + d -1$. Then there is an isotrivial degeneration of $X$ to $X_0$ (over some pointed curve) such that $X_0$ is isomorphic to the hypersurface defined by $(x_0 ^d = x_1 \dots x_d)$ in $\P^n$. Thus, \Cref{mainconj2intro} holds for all such $X$.
\end{thm}

 \subsection{Proof strategies} The strategy for proving \Cref{mainthm3intro} is to use the theory of $F$-adjunction (\cite{SchwedeFAdjunction},\cite{CasciniGongyoSchwedeUniformBounds}, \cite{CTW17}), along with the techniques from the minimal model program. The key idea is that the condition $\alpha (X, \Delta) < 1$ on a log Fano pair $(X, \Delta)$ implies the existence of a \emph{special divisor} over $(X, \Delta)$ (\cite{BLX22}), which by a result of Zhuang, provides a lower-dimensional log Fano pair $(E, \Delta_E)$ on a birational model of $X$. Then, the proof proceeds by a careful reduction modulo $p$ argument and lifting $F$-splittings from $(E, \Delta_E)$. See Section \ref{section:inductivestep} for the details.
 
 The same idea also allows us to show \Cref{introthm:Fregularity}, where we deduce global $F$-regularity from the results of Watanabe \cite{WatanabeFregularFpureRings} which deal with the case of pairs on $\PP^1$. This is discussed in Section \ref{section:Fregularity}.

Constructing \emph{normal} toric degenerations as in \Cref{mainthm2intro} involves exploiting the projectivity of appropriate moduli spaces of degree $d$ hypersurfaces along with a result of Starr from \cite{Sta06}. Starr shows that when the degree of a smooth hypersurface $X \subset \P^n $ is very low relative to $n$, then any general hypersurface of degree $d$ in $\P^d$ can be realized as a linear section of $X$ (up to isomorphism). Since $X$ degenerates to the cone over any of its linear sections, \Cref{mainthm2intro} follows from a limiting argument. In Section \ref{section:hypersurfaces}, we provide two different proofs using the GIT and K-moduli spaces respectively.

%leading to natural questions such as: \textit{how does $\s_p (R)$, the $F$-signature of the reduction of $R$ to characteristic $p$ depend on $p$? when is it independent of $p$? does the limit of $\s_p (R)$ as $p \to \infty$ exist? For $p \gg 0$, does $\s_p (R)$ have an interpretation in terms of the singularities of $R$ (over $\C$)?}

\subsection*{Acknowledgements}
We would like to thank Harold Blum, Anna Brosowsky, James Hotchkiss, Seungsu Lee, Devlin Mallory, Mircea Musta\c t\u a, Karl Schwede, Karen Smith, Sridhar Venkatesh, Kevin Tucker and Ziquan Zhuang for valuable conversations. YL was partially supported by the NSF grant DMS-2237139 and an AT\&T Research Fellowship from Northwestern University.
SP was partially supported by the NSF grants \#1952399, \#1801697 and \#2101075.
We also thank the departments at the University of Michigan, Northwestern University and the University of Utah for their hospitality during various visits by the authors.
Part of this material is based upon work supported by the National Science Foundation under Grant No. DMS-1928930 and by the Alfred P. Sloan Foundation under grant G-2021-16778, while the author was in residence at the Simons Laufer Mathematical Sciences Institute (formerly MSRI) in Berkeley, California, during the Spring 2024 semester.

\section{Preliminaries}

\subsection{Singularities of MMP and Negativity Lemma}
We recall the relevant notions of singularities of the MMP. Throughout this subsection a pair $(X, \Delta)$ will refer to a normal variety over an algebraically closed field $k$ and an effective $\Q$-divisor $\Delta $ such that some multiple of $K_X + \Delta$ is Cartier, where $K_X$ denotes a canonical divisor of $X$. 

\begin{defi}
    A divisor $E$ \emph{over} $X$ refers to some prime divisor $E \subset Y \xrightarrow{\pi}  X$, where $Y$ is a normal variety over $k$ and $\pi: Y \to X$ is a projective, birational morphism to $X$. The \emph{log discrepancy} of such a divisor $E$ with respect to the pair $(X, \Delta)$ is a rational number defined by
    \[ A_{(X, \Delta)} (E) = \ord_E (K_Y + \Delta_Y  + E - \pi^* (K_X + \Delta_X) ) \]
    where $\Delta_Y$ is the strict transform of $\Delta$ and $K_Y$ is a canonical divisor of $Y$ that is compatible with $K_X$ on the locus where $\pi$ is an isomorphism.
\end{defi}

We use log discrepancies to define the following class of singularities.

\begin{defi} \label{def:KLT}
    A pair $(X, \Delta)$ is defined to be
    \begin{enumerate}
        \item \emph{Kawamata log terminal (KLT)} if $A_{(X, \Delta)} (E) > 0$ for every divisor $E$ over $X$.

        \item \emph{log canonical (LC)} if for every divisor $E$ over $X$, we have $A_{(X, \Delta)} (E) \geq 0$. Any divisor $E$ over $X$  with $A_{(X, \Delta)} (E) = 0$ is called an \emph{LC-place} of $(X, \Delta)$.

        \item \emph{purely log terminal (PLT)} if $(X, \Delta)$ is log canonical, $S = \lfloor \Delta \rfloor $ is a prime divisor on $X$ and $S$ is the only LC-place of $(X, \Delta)$.
    \end{enumerate}
\end{defi}

We also have the following global versions for projective varieties:

\begin{defi} \label{def:logFano}
    A pair $(X, \Delta)$ is called a \emph{log Fano} pair if:
    \begin{itemize}
        \item $X$ is a normal, projective variety,
        \item $(X, \Delta)$ is KLT and
        \item the $\Q$-divisor $-K_X -\Delta$ is ample.
    \end{itemize}
    A projective variety $X$ is said to be of \emph{Fano type} if there exists some effective $\Q$-divisor $\Delta$ such that $(X, \Delta)$ is a log Fano pair.
\end{defi}

The following numerical invariants of singularities will play a central role in this paper:

\begin{defi} \label{def:complexalpha}
    Given a log canonical pair $(X, \Delta)$ and an effective $\Q$-Cartier divisor $D \geq 0$ on $X$, the \emph{log canonical threshold} of $D$ with respect to $(X, \Delta)$ is defined to be
    \[ \lct (X, \Delta; D) := \sup \{ t \geq 0 \, | \, (X, \Delta + tD) \, \text{is log canonical} \}.\]
    Globally, given a log Fano pair $(X, \Delta)$, the \emph{$\alpha$-invariant} or the \emph{global log canonical threshold} of $(X, \Delta)$ is defined to be the infimum:
    \[ \alpha (X, \Delta) := \inf \{ \lct (X, \Delta; D) \ | \ \text{$D \geq 0$ is any $\Q$-divisor with} \,  -(K_X + \Delta) \sim _\Q D\}.  \]
    
\end{defi}

We will also repeatedly use the negativity lemma, which holds over algebraically closed fields of arbitrary characteristic:

\begin{lem} \cite[Section 2.3]{Birkar16} Let $f: Y \to X$ be a projective, birational morphism between normal projective varieties over $k$ and suppose we have an $f$-exceptional, $\Q$-Cartier divisor $E $ on $Y$. If $E$ is $f$-nef, then $E \leq 0$. In particular, if $E$ is numerically trivial over $X$, then $E = 0$.
    
\end{lem}
As a consequence, we have:
\begin{lem} \label{lem:birationalcontraction}
    Let $Y$ and $ Y'$ be normal projective varieties over an algebraically closed field $k$, and $f: Y \to Y' $ be a birational contraction. Let $\Delta \geq 0$ be an effective $\Q$-divisor on $Y$ and let $\Delta ' = f_{*} \Delta$ be the corresponding divisor on $Y'$. Assume that $K_Y +\Delta$ and $K_{Y'} + \Delta'$ are  both $\Q$-Cartier. Then, we have
    
    \begin{enumerate}
        \item If $K_Y 
 + \Delta \sim 0$, then $K_{Y'} + \Delta ' \sim 0$ as well. Moreover, in this case, if $(Y, \Delta)$ is KLT then so is $(Y', \Delta')$.

        \item   Assume that $K_Y + \Delta$ is nef over $Y'$. Then, 
\begin{equation}\label{eq:negativity}
K_Y+\Delta \leq  f^*(K_{Y'}+\Delta').
\end{equation}
In this case, if $(Y', \Delta')$ is KLT, then so is $(Y, \Delta)$.
    \end{enumerate}
\end{lem}

\begin{proof}
    Write
    \begin{equation} \label{restrictiontoU}
        K_Y + \Delta + E = f^* (K_{Y'} + \Delta') 
    \end{equation}
    for some $f$-exceptional divisor $E$ on $Y$. Moreover, let $\mu: Z \to Y$ be a log resolution of $(Y, \Delta + E)$. Then, note that for any prime divisor $F \subset Z$, we have $A_{Y, \Delta +E} (F) = A_{Y', \Delta'} (F)$. 
    
    \begin{enumerate}
        \item We now apply the negativity lemma for $f$. Thus, if $K_Y + \Delta \sim 0$, $E$ is numerically trivial over $Y'$ and hence must be zero. This proves part (1).

        \item  For part (2), note that if $K_Y + \Delta $ is nef over $Y'$, then $-E$ is nef over $Y'$ which means that $E$ must be effective. Since $E$ is effective, we also have $A_{Y, \Delta} (F) \geq A_{Y, \Delta + E} (F) = A_{Y', \Delta'} (F)$ for any prime divisor $F \subset Z$. This proves $(2)$. \qedhere
    \end{enumerate}
\end{proof}

\subsection{$F$-signature}
	
 For any prime number $p$, let $R$ be a ring of characteristic $p$. For each $e \geq 1$, we have the $R$-algebra $\Fe R$ obtained by the pushforward of $R$ along the Frobenius map $F^e: \Spec(R) \to \Spec(R)$. Concretely, $F_{*} ^e R$ is the same as $R$ as a ring, but the $R$-module action is given by:
 $$ r\cdot F_* ^e s := F_{*} ^e (r^{p^e} s) \textrm{\quad for any $r\in R$ and $F_* ^e s \in F_* ^e R$}   .$$

 	Now let $(R, \fm)$ denote a normal local ring of Krull dimension $d>0$ and $X$ denote the normal scheme $\Spec(R)$. Throughout, we will assume that $R$ is the localization of a finitely generated $k$-algebra at a maximal ideal,  which also makes it  $F$-finite (i.e., $F_{*} ^e R$ is a finitely generated $R$-module for any $e \geq 1$), with the rank of $\Fe R $ over $R$ being $p^{ed}$. Let $\Delta$ be an effective $\Q$-divisor on $X = \Spec(R)$. Since $\Delta$ is effective, we have a natural inclusion $R \subset R( \lceil (p^e -1) \Delta \rceil)$ of reflexive $R$-modules. Here, $R(\lceil (p^e -1) \Delta \rceil) $ denotes the $R$-module corresponding to the reflexive sheaf $\cO_X (\lceil (p^e -1) \Delta \rceil)$. Thus, applying $\Hom_R ( \, \textunderscore \, , \, R)$ to the natural inclusion $\Fe R \subset \Fe (R(\lceil (p^e -1) \Delta \rceil))
  $, we get
  \[   \Hom_R \big( \Fe R(\lceil (p^e -1) \Delta \rceil), R \big) \subset \Hom _R (\Fe R, R).  \]
Thus, given any element  $\varphi \in \Hom_R \big( \Fe R(\lceil (p^e -1) \Delta \rceil), R \big) $, it can be naturally viewed as a map $\varphi: \Fe R \to R$.
 	
 	\begin{defi}[Splitting ideals and free rank] \label{frrkkdfn}
 	For any $e \geq 1$, we define the subset $I_e ^{\Delta} \subseteq R$ as 
 	\[I_e ^{\Delta} = \left\{x\in R \mid \varphi(\Fe x) \in \fm \,\text{for every map } \varphi \in \Hom_R \big( \Fe R(\lceil (p^e -1) \Delta \rceil), R \big) \, \right \}.\]
 		We observe that $I_e ^{\Delta}$ is an ideal of finite colength in $R$
 	and we call 
 	 $$ a_e ^{\Delta} = \ell_{R}(R/I_{e} ^{\Delta}) $$
 	the \emph{$\Delta$-free rank} of $\Fe R$, where $\ell_R$ denotes the length as an $R$-module.
 	\end{defi}

	\begin{defi} \cite[Theorem 3.11, Proposition 3.5]{BlickleSchwedeTuckerFSigPairs1} \label{F-sigdfn}
 		Let $(R, \Delta)$ be a pair as above, and $a_e ^{\Delta} (R)$ denote the $\Delta$-free rank of $\Fe R$ (\Cref{frrkkdfn}). Then the \emph{$F$-signature} of $(R, \Delta)$ is defined to be the limit:
 		$$ \s(R, \Delta) := \lim_{e \to \infty} \frac{a_{e} ^{\Delta}}{p^{ed}} $$
 		where $d$ is the Krull dimension of $R$. This limit exists by \cite{BlickleSchwedeTuckerFSigPairs1}.
 	\end{defi}

\begin{defi} \label{Fsigofsectionrings}
    Let $(X, \Delta)$ be a normal, projective pair over $k$ of positive dimension and let $L$ be an ample divisor over $X$. Let $(S, \fm)$ denote the $\N$-graded section ring of $X$ with respect to $L$:
    \[ S(X, L) = \bigoplus _{m \geq 0} H^0 (X, \cO (mL)). \]
    Moreover, let $\tilde{\Delta}$ denote the cone over $\Delta$ with respect to $L$, which is a $\Q$-divisor over $\Spec(S)$.
    
   Then we define the $F$-signature of $(X, \Delta)$ with respect to $L$ as:
  \[ \s(X, \Delta;L) = \s(S_\fm, \tilde{\Delta}),\]
   where $S_\fm$ denotes the localization of $S$ at $\fm$ and the $F$-signature in the local setting is as defined in \Cref{F-sigdfn}.

   Let $(X, \Delta)$ be a log Fano pair over $k$ and $r$ denote a positive integer divisible by the index of $(X, \Delta)$. Let $L$ denote the ample divisor $-r(K_X + \Delta)$. Then, the \emph{$F$-signature} of $(X, \Delta)$ is defined to be
    $$ \s(X, \Delta) := r \, \s (X, \Delta; L) $$
    where $\s$ denotes the $F$-signature of pairs, as defined in \Cref{Fsigofsectionrings}.
\end{defi}

\subsection{$F$-splitting, $F$-regularity and global $F$-adjunction} Let $k$ denote a perfect field of characteristic $p>0$.

\begin{defi}[Sharp $F$-splitting] \cite[Definition 3.1]{SchwedeSmithLogFanoVsGloballyFRegular} \label{defnsharpFsplitting}  Let $X$ be a normal variety over $k$ and $\Delta \geq 0$ be an effective $\Q$-divisor. The pair $(X, \Delta) $ is said to be \emph{globally sharply $F$-split} (resp. \emph{locally} sharply $F$-split) if there exists an integer $e > 0$ such that the natural map 
\begin{equation*} 
\cO_X \to F^e_* \cO_X( \lceil (p^e -1) \Delta \rceil)
 \end{equation*}
splits (resp. splits locally) as a map of $\cO_{X}$-modules.
A normal variety $X$ is said to globally $F$-split if the pair $(X, 0)$ is globally sharply $F$-split. 
\end{defi}

\begin{defi}[$F$-regularity]\label{defn.GlobFreg} \cite[Definition 3.1]{SchwedeSmithLogFanoVsGloballyFRegular} Let $X$ be a normal variety over $k$ and $\Delta \geq 0$ be an effective $\Q$-divisor. The pair $(X, \Delta) $ is said to be \emph{globally $F$-regular} (resp. locally strongly $F$-regular) if for any effective Weil divisor $D$ on $X$, there exists an integer $e \gg 0$ such that the natural map 
\begin{equation*} 
\cO_X \to F^e_* \cO_X( \lceil (p^e -1) \Delta \rceil + D)
 \end{equation*}
splits (resp. splits locally) as a map of $\cO_{X}$-modules.
A normal variety $X$ is said to globally $F$-regular if the pair $(X, 0)$ is globally $F$-regular. Similarly, a ring $R$ is called strongly $F$-regular if the pair $(\Spec(R), 0)$ is locally (equivalently, globally) strongly $F$-regular.
\end{defi}

\begin{rem} When $X= \Spec(R)$ is an affine variety and $\Delta$ is an effective $\Q$-divisor, the pair $(X, \Delta)$ being globally $F$-regular (resp. globally sharply $F$-split) is equivalent to the pair $(R, \Delta)$ being locally strongly $F$-regular (resp. locally sharply $F$-split) \cite{SchwedeSmithLogFanoVsGloballyFRegular}.
\end{rem}

\begin{rem}
    A central result of Schwede and Smith is that if $X$ is a globally $F$-regular variety over an algebraically closed field $k$, then $X$ is of Fano type \cite{SchwedeSmithLogFanoVsGloballyFRegular}. Moreover, there exists a log Fano pair $(X,\Delta)$ such that $(X, \Delta)$ is globally $F$-regular.
\end{rem}

\begin{rem}\label{FregularityremarK}
A local domain $R$ is strongly $F$-regular if and only if its $F$-signature $\s(R)$ is positive \cite{AberbachLeuschke}. More generally, a pair $(R, \Delta)$ (where $R$ is normal, local) is strongly $F$-regular if and only if the $F$-signature $\s(R, \Delta)$ (\Cref{F-sigdfn}) is positive \cite[Theorem 3.18]{BlickleSchwedeTuckerFSigPairs1}.
\end{rem}

\begin{lem} \label{lem:GFRunderbirationalmap} 
 Let $f: Y \dashrightarrow Y'$ be a birational map between two normal projective varieties over a perfect field $k$ of characteristic $p>0$. Let $\Delta$ be an effective $\Q$-divisor over $Y$ and let $\Delta' = f_* \Delta$.
\begin{enumerate}
\item Assume that $f$ is an isomorphism in codimension $1$, i.e., there exist open subsets $U \subset Y$ and $U' \subset Y'$ such that the codimension of $Y \setminus U$ and $Y' \setminus U'$ is at least two in $Y$ and $Y'$ respectively, and such that $f$ restricts to an isomorphism $f|_U : U \to U'$. Then, $Y$ is globally $F$-regular if and only if $Y'$ is globally $F$-regular. Moreover, if $(Y, \Delta)$ is a globally $F$-regular pair, then so is $(Y', f_* \Delta)$.

\item Suppose $f$ is a birational morphism and $(Y, \Delta$) is globally $F$-regular. Then, $(Y', f_* \Delta) $ is globally $F$-regular.

\item Suppose that $K_Y+\Delta \leq f^*(K_{Y'}+\Delta')$. Then, $(Y',\Delta')$ is globally $F$-regular  if and only if $(Y,\Delta)$ is.
\end{enumerate}
\end{lem}

\begin{proof}
We first check the following well-known fact, but we include a proof for completeness. If $Y$ is a normal variety and $U \subset Y$ is a big open subset, i.e., $U$ is a non-empty open subset of $Y$ such that the codimension of $Y \setminus U$ in $Y$ is at least two. Then, $(Y, \Delta)$ is globally $F$-regular if and only if $(U, \Delta|_U)$ is globally $F$-regular. Moreover, the only if direction holds for any open subset $U \subset Y$. Suppose $D_U$ is an effective Weil-divisor on $U$ and let $D$ be its closure in $X$. Then for any $e$ such that the map $\cO_X \to F_* ^e \cO_X( \lceil (p^e -1) \Delta \rceil + D)$ splits (which exists if $(X, \Delta)$ is globally $F$-regular), the restriction $\cO_U \to F_* ^e \cO_U( \lceil (p^e -1) \Delta|_U \rceil + D_U) $ also splits.

        Conversely, suppose $(U, \Delta|_U)$ is globally $F$-regular and let $D$ be a fixed effective Weil divisor on $X$. Consider an $e \gg 0 $ and a splitting $\varphi: F_* ^e \cO_U( \lceil (p^e -1) \Delta|_U \rceil + D|_U) \to \cO_U$. Since both sheaves are reflexive, the splitting $\varphi$ extends to a map on $X$ giving a splitting of the map $\cO_X \to F_* ^e \cO_X (\lceil (p^e -1 ) \Delta \rceil + D)$.
    \begin{enumerate}
        \item The first statement is proved in \cite[Proposition 6.3]{SchwedeSmithLogFanoVsGloballyFRegular}.
        To prove the second statement in the lemma using the above observation, we may replace $Y'$ by $U'$ and the lemma follows since $f$ restricts to an isomorphism over $U'$ by assumption.

        \item Since $f$ is a birational morphism, there is a big open set $W $ of $Y'$ such that $f$ restricts to an isomorphism over $W$. Now, the global $F$-regularity of $(W, \Delta' |_W)$ follows from the global $F$-regularity of $(Y, \Delta)$. 

        \item This is proved in \cite[Proposition 2.11]{HaconXuThreeDimensionalMinimalModel}. \qedhere
        \end{enumerate}
\end{proof}

\begin{defi}
    A pair $(X, \Delta)$ is said to be globally \emph{purely} $F$-regular if $S = \lfloor \Delta \rfloor $ is a prime divisor and if for any effective Weil divisor $D$ not containing $S$ in the support, there exists an integer $e \gg 0$, such that, the natural map 
\begin{equation*} 
\cO_X \to F^e_* \cO_X( \lceil (p^e -1) \Delta \rceil + D)
 \end{equation*}
splits as a map of $\cO_{X}$-modules.
\end{defi}

The following lemma is a key tool for the inductive arguments in Section 6.
\begin{lem} \cite[Lemma 2.7]{CTW17} \label{lem:Fadjunction}
    Let $(X, \Delta = S + B)$ be a pair such that $X$ is projective over $k$, $S$ is a prime divisor on $X$ and $\lfloor B \rfloor  = 0$. Assume that $-(K_X + S + B)$ is ample.
    Let $\nu: S' \to S$ be the normalization map and $B_{S'}$ be defined by adjunction (see \cite{das_different_different_different}):
    \[ K_{S'} + B_{S'}  = \nu^*((K_X + S+ B)|_S).  \]
    Then, the pair $(S', B_{S'})$ is globally $F$-regular if and only if the pair $(X, S + B)$ is globally purely $F$-regular. And in that case, $S$ is already normal and hence $\nu$ is an isomorphism. 
\end{lem} 

\begin{proof}
    \cite[Lemma 2.7]{CTW17} proves this statement in the globally $F$-split case under the additional assumption that $(p^e -1) (K_X + S +B)$ is a Cartier divisor for some $e >0$. The $F$-regular case follows from this by the following well-known criterion: the pair $(S', B_{S'})$ is globally $F$-regular if and only if for any effective divisor $H \geq 0 $, there exists $\epsilon > 0$ (depending on $H$) and another divisor $H' \geq  \epsilon H$  such that $(p^e -1) (K_{S'} + B_{S'} + H')$ is Cartier for some $e >0$ and $(S', B_{S'} + H')$ is globally $F$-split.
\end{proof}

\subsection{The $\FA$-invariant.}
Now we recall a Frobenius analog of the $\alpha$-invariant for log Fano pairs.

\begin{defi}{\cite[Lemma 3.7]{PandeFrobeniusAlpha}}  \label{alphadefinitions}
 Let $(X, \Delta)$ be a projective, globally $F$-regular pair and $L$ be an ample Cartier divisor on $X$. Then, $\FA(X, \Delta; L)$ is equal to the supremum of any of the following sets:
 \begin{enumerate}
     \item The set of rational numbers $\lambda \geq 0$ such that the pair $(S, \tilde{\Delta} + \frac{\lambda}{n} D)$ is sharply $F$-split for every $n \in \N$ and every effective divisor $D \sim nL$.
     \item The set of $\lambda \geq 0$ such that the pair $(S, \tilde{\Delta} + \frac{\lambda}{n} D)$ is strongly $F$-regular for every $n \in \N$ and every effective divisor $D \sim nL$.
     \item The set of $\lambda \geq 0$ such that the pair $(X, \Delta +\frac{\lambda}{n} D)$ is globally sharply $F$-split for every $n \in \N$ and every effective divisor $D \sim nL$.
     \item The set of $\lambda \geq 0$ such that the pair $(X, \Delta+\frac{\lambda}{n} D)$ is globally $F$-regular for every $n \in \N$ and every effective divisor $D \sim nL$.
     \item The set of $\lambda \geq 0$ such that the pair $(X, \Delta + \lambda D)$ is globally $F$-regular for every effective $\Q$-divisor $D \sim_\Q L$.
 \end{enumerate}
\end{defi}

\begin{defi} \label{FAXdfn}
    Let $(X, \Delta)$ be a globally $F$-regular log Fano pair over $k$ and $r$ be a positive integer divisible by the index of $(X, \Delta)$. Let $L = -r(K_X + \Delta)$ denote the corresponding ample divisor. Then, the $\FA$-\emph{invariant} of $(X, \Delta)$ is defined to be
    $$ \FA (X, \Delta) := r \, \FA (X, \Delta; L ) $$
    where $\FA$ denotes the $\FA$-invariant as defined in Definition~\ref{alphadefinitions}.
\end{defi}

\begin{rem}
    The $\FA$-invariant of any globally $F$-regular log Fano pair $(X, \Delta)$ is positive \cite[Theorem 4.6]{PandeFrobeniusAlpha}. Moreover, it is shown in \cite[Theorem 5.5]{PandeFrobeniusAlpha} that in this case we have $\FA(X, \Delta)  \leq 1/2$.
\end{rem}

\begin{rem}
    The motivation relating the $\FA$-invariant and the complex $\alpha$-invariant (\Cref{def:complexalpha}) comes from the following theorem due to Hara and Yoshida \cite{HaraYoshidaGeneralizationOfTightClosure}, and Takagi \cite{TakagiInterpretationOfMultiplierIdeals}: Let $(X, \Delta)$ be a complex log Fano pair and $D \sim _\Q -(K_X + \Delta)$ be an effective divisor over $X$ and let $0 < t < 1$ be a fixed rational number. Then $(X, \Delta + tD)$ is KLT if and only if for all $p \gg 0$, the characteristic $p$ models $(X_p, \Delta_p + t D_p)$ are globally $F$-regular.
\end{rem}

\section{Lim-inf $F$-signature}

In this section, we will define the lim-inf $F$-signature of a complex KLT singularity and formulate the precise versions of the conjecture on the positivity of the lim-inf $F$-signature. We will focus on the case of log Fano pairs and show that the global case implies the local case of KLT singularities (\Cref{thm:localtoglobalreduction}).

\subsection{Definition of the lim-inf $F$-signature}
\begin{defi}[Reduction to characteristic $p \gg 0$ of a log Fano pair] \label{reductindfn}
    Let $(X, \Delta)$ be a log Fano pair over $\C$. A \emph{family of characteristic $p \gg 0$ models} of $(X, \Delta)$, denoted by $(X_A, \Delta_A) \to \Spec(A)$, consists of the following data:
    \begin{enumerate}
        \item A finitely generated smooth $\Z$-algebra $A \subset \C$.
        \item A normal projective scheme $X_A$, that is flat over $\Spec(A)$ and such that the fiber $X_A \times_{A} \Spec(\C) $ is  isomorphic to $X$. Let $\phi: X_A \times_A \C \to X$ be an isomorphism.
        \item An effective $\Q$-Weil divisor $\Delta_A$ on $X_A$ that is flat over $A$ and such that we have  $\phi(\Delta _A \times_A \C) = \Delta$. We also assume that $K_{X_A} + \Delta_{A}$ is $\Q$-Cartier.
        \item Assume that $-(K_{X_A} + \Delta_A)$ is ample over $\Spec(A)$.
    \end{enumerate}
    Given a family of characteristic $p \gg 0$ models of $(X, \Delta)$, for any prime $p \gg 0$, the \emph{reduction to characteristic $p$} of $(X, \Delta)$ refers to a fiber
    \[ (X_{\fm}, \Delta_\fm) :=  (X_A \times_A \Spec(A/\fm), \Delta_A \times_A \Spec(A/\fm) ),\] where $\fm \subset A$ is any maximal ideal whose residue characteristic is $p$. Recall that in that case $A/\fm$ is a finite field of characteristic $p$.
\end{defi}  

\begin{rem}
    Given $(X_A, \Delta_A)$, a family of characteristic $p \gg 0$ models of a log Fano pair $(X, \Delta)$, we may assume, using \cite[Theorem~5.1]{SchwedeSmithLogFanoVsGloballyFRegular} that (after possibly replacing $A$ by a finitely generated extension) for all maximal ideals $\fm \subset A$, the fiber $(X_{\fm}, \Delta_\fm)$ is a globally $F$-regular log Fano pair. In particular, we may assume that all fibers of $ X_A \to \Spec(A)$ are geometrically normal. \end{rem}

Let $(X, \Delta)$ be a log Fano pair over $\C$ and let $(X_A, \Delta_A) \to \Spec(A)$ be a family of characteristic $p \gg 0$ models of $(X, \Delta)$. Note that if the ring $A$ is not a one dimensional ring, then the reduction to characteristic $p \gg 0$ of $(X, \Delta)$ over $A$ is not necessarily uniquely defined. Rather, for all prime numbers $p \gg 0$, we get a family of globally $F$-regular log-Fano pairs $ (X _{p}, \Delta_p):= (X_A \times_{A} \Spec{(A/pA)},  \Delta |_{X_p})$ over $ \Spec(A/pA)$ where $\Delta_p$ denotes the restriction of $\Delta$ to the normal subscheme $X_p$. For any maximal ideal $\fm$ of $A$ that lies over $p$, we call the fiber $X_{\kappa(\fm)} = X \times_{A} \Spec{(A/\fm)}$ a \emph{characteristic $p$ model} of $X$.

\begin{defi}[The lim-inf $F$-signature and $\FA$-invariant of complex log Fano pairs] \label{liminfFsigdfn}
    Let $(X, \Delta)$ be a log Fano pair over $\C$. Consider a family of characteristic $p$ models $(X_A, \Delta_A) \to \Spec(A)$ of $(X, \Delta)$. Then, for each prime number $p \gg 0$,
    \begin{itemize}
        \item we define the \emph{$F$-signature in characteristic $p$}, denoted by $\s_{p}(X, \Delta)$ as
    $$ \s_{p}(X, \Delta) \,  := \sup_{\fm \subset A/pA \,  \text{maximal} } \s(X_{\fm}, \Delta_\fm). $$
    And we define the \emph{lim-inf $F$-signature} of $(X,\Delta)$ to be:
    $$ \s_{-} (X, \Delta) = \liminf_{p \to \infty} \s_{p}(X, \Delta) .$$

    \item Similarly, we define the \emph{$\FA$-invariant in characteristic $p$}, denoted by $\alpha_{p}(X, \Delta)$ as
    $$ \alpha_{p}(X, \Delta) \,  := \sup_{ \fm \subset A/pA \, \text{maximal} } \alpha_{F}(X_{\fm}, \Delta_\fm). $$
     And we define the \emph{lim-inf $\alpha_F$-invariant} of $(X, \Delta)$ to be:
    $$ \alpha_{-} (X, \Delta) = \liminf_{p \to \infty} \alpha_{p}(X, \Delta).$$    
    \end{itemize}
\end{defi}

%\begin{defi}
 %   Let $X$ be a $\Q$-Fano variety over $\C$. Let $s_{p}(X)$ denote the characteristic $p$ $F$-signature (computed using any family of models of models for $X$). 
%\end{defi}

%\begin{defi}
 %   Let $X$ be a $\Q$-Fano variety over $\C$ and $X_A \to \Spec(A)$ be a family of models for $X$ over a smooth $\Z$-algebra $A$. Let $\alpha_{p}(X_A)$ denote the characteristic $p$ $\alpha_{F}$-invariant of $X_A$.
%\end{defi}

Note that by definition, the lim-inf $F$-signature and the $\FA$-invariant are both computed using a chosen family of characteristic $p$-models of $(X,\Delta)$. We now check that these numerical invariants are indeed independent of the models chosen.

\begin{lem}
    Let $(X, \Delta)$ be a log Fano pair over $\C$. Consider two families of characteristic $p$ models $(X_A, \Delta_A) \to \Spec(A)$ and $(X_B, \Delta_B) \to \Spec(B)$ of $(X, \Delta)$ over finitely generated smooth $\Z$-algebras $A$ and $B$. Then, for any prime $p \gg 0$,
    \begin{enumerate}
        \item we have
\[ \sup_{ \fm \subset A/pA \,  \text{maximal} } \s(X_{\fm}, \Delta_\fm)  = \sup_{ \mathfrak{n} \subset B/pB \,  \text{maximal} } \s(X_{\mathfrak{n}}, \Delta_\mathfrak{n}). \]
        \item Similarly,
        \[  \sup_{ \fm \subset A/pA \, \text{maximal} } \alpha_{F}(X_{\fm}, \Delta_\fm) = \sup_{ \mathfrak{n} \subset B/pB \, \text{maximal} } \alpha_{F}(X_{\mathfrak{n}}, \Delta_\mathfrak{n}).  \]
    \end{enumerate}
\end{lem}

\begin{proof}
By finding a common enlargement for $A$ and $B$ if necessary, it is sufficient to prove the lemma when $A $ is a subring of $B$. Furthermore, by \cite[\href{https://stacks.math.columbia.edu/tag/00TF}{Tag 00TF}]{stacks-project} and generic smoothness (applied over $\text{Frac}(A)$, which is a field of characteristic zero), there exist localizations $A_f$ and $B_g$ such that $A_f \subset B_g$ and the induced map $\Spec(B_g) \to \Spec(A_f)$ is a smooth map of schemes. Therefore, it is sufficient to prove the lemma when $\Spec(B) \to \Spec(A)$ is a dominant and smooth map of smooth schemes over $\Z$.

Fix a prime $p \gg 0$ such that the induced map $ \Spec(B/pB) \to \Spec(A/pA)$ is dominant and smooth. Let $\Spec(A_p)$ and $\Spec(B_p)$ be fixed connected components of $\Spec(A/pA) $ and $\Spec(B/pB)$ respectively, such that there is an induced dominant map $\Spec(B_p) \to \Spec(A_p)$. Note that $A_p$ and $B_p$ are regular domains of finite type over $\mathbb{F}_p$. Moreover, we may assume that the map between the fraction fields
\[   K:= \text{Frac}(A_p) \subset L:=\text{Frac}(B_p)  \]
is separable.
\begin{enumerate}
    \item  First we claim that
\begin{equation} \label{supequalsgeneric1}   \sup_{ \fm \subset A_p \,  \text{maximal} } \s(X_{\fm}, \Delta_\fm) = \s(X_{K}, \Delta_{K}) = \s(X_{K^\infty}, \Delta_{K^\infty}), \end{equation}
and similarly,
\begin{equation} \label{supequalsgeneric2} \sup_{\mathfrak{n} \subset B_p \,  \text{maximal}} \s(X_{\mathfrak{n}} , \Delta_{\mathfrak{n}}) =\s (X_{L}, \Delta_{L}) = \s (X_{L^\infty}, \Delta_{L^\infty}) .  \end{equation}
By symmetry, it is enough to prove the claim for $A_p$. To see this, note that since $A_p$ is regular, each maximal ideal $\fm \subset A_p$ is a complete intersection (i.e., generated by a regular sequence). Thus, by \cite[Corollary 6.1]{TaylorAdjunctionforFsig}, for any maximal ideal $\fm \subset A_p$, we have $\s(X_{B_{\fm}}, \Delta_{B_\fm}) \geq \s(X_{\fm}, \Delta_\fm)$. Moreover, by the lower semicontinuity of the $F$-signature \cite[Theorem 5.7]{PolstraSemicontinuityofFsig}, we have $\s(X_K, \Delta_K) \geq \s(X_{B_{\fm}}, \Delta_{B_\fm})$. Lastly, \cite[Theorem~3.6]{CRSTBertiniTheorems} implies that 
\[ \s(X_{K}, \Delta_{K}) = \s(X_{K^\infty}, \Delta_{K^\infty}) \]
since the induced map on the section rings of $X_{K^\infty} $ and $ X_K$ has a regular fiber (isomorphic to $\Spec(K^\infty)$).
This shows that the RHS is at least the LHS in \Cref{supequalsgeneric1} and \Cref{supequalsgeneric2}.

For the reverse inequality, we use \cite[Theorem 4.10]{CRSTBertiniTheorems} to see that for any $0 < \lambda < \s(X_{K^{\infty}}, \Delta_{K^\infty})$, there exists a dense open set $U_{\lambda} \subset \Spec(A_p)$ such that for any maximal ideal $\fm \in U_{\lambda}$, we have $\s(X_{\fm}, \Delta_{\fm}) > \lambda$. Then, the reverse inequality follows by taking $\lambda \to \s(X_{L^{\infty}}, \Delta_{L^\infty}).$ This completes the proof of the claim.

Now, to prove Part (1) of the lemma using the above claim, it suffices to show that $\s(X_K, \Delta_K) = \s(X_L, \Delta_L)$. But this follows again from \cite[Theorem 3.6]{CRSTBertiniTheorems}.

\item Arguing as in the proof of Part (1), we first show that
\begin{equation} \label{supequalsgeneric3} \FA(X_{K^\infty}, \Delta_{K^\infty}) =  \sup_{ \fm \subset A_p \,  \text{maximal} } \FA(X_{\fm}, \Delta_\fm) . \end{equation}

First suppose that $\FA(X_{K^\infty}, \Delta_{K^\infty}) = 0$. Then, by \cite[Theorem~3.10]{PandeFrobeniusAlpha}, $(X_{K^\infty}, \Delta_{K^\infty})$ is not globally $F$-regular. By \Cref{supequalsgeneric1}, we know that $(X_\fm, \Delta_\fm)$ is not globally $F$-regular for any maximal ideal $\fm \subset A_p$. Therefore, by \cite[Theorem~3.10]{PandeFrobeniusAlpha} again, the LHS of \Cref{supequalsgeneric3} is zero as well. Therefore, we may assume that the locus
\[ U = \{ y \in \Spec(A_p) \, | \, (X_{y^\infty}, \Delta_{y^\infty}) \, \text{is globally $F$-regular} \} \]
is non-empty. But this locus is also open in $\Spec(A_p)$. Therefore, it is sufficient to prove \Cref{supequalsgeneric3}  when $U = \Spec(A_p)$. This follows immediately from the lower semi-continuity of the $\FA$-invariant proved in \cite[Theorem~V.2.1]{Pandethesis}.

Finally, we use \cite[Corollary 4.14]{PandeFrobeniusAlpha} to conclude that
\[ \FA(X_{K^\infty}, \Delta_{K^\infty}) = \FA(X_{L^\infty}, \Delta_{L^\infty}). \]
This completes the proof of the lemma. \qedhere
\end{enumerate}
\end{proof}

\subsection{Statements of the main conjectures.}
We now formulate two precise versions of the question first raised by Carvajal-Rojas, Schwede and Tucker \cite[Question 5.9]{CarvajalRojasSchwedeTuckerFundamentalGroups} using the lim-inf $F$-signature and the lim-inf $\FA$-invariant introduced above (\Cref{liminfFsigdfn}):
\begin{conj}\label{conj:alpha-lim-positive}
    Let $(X,\Delta)$ be a KLT log Fano pair over $\bC$. Then, both $\s_{-} (X, \Delta)$ and $\alpha_-(X,\Delta)$ are positive.
\end{conj}
There is a natural strengthening of the above conjecture that takes into account the $F$-signature of the family of characteristic $p$ reductions of $(X, \Delta)$.
\begin{conj} \label{strongerconjversion}
Let $(X, \Delta)$ be a KLT log Fano pair over $\C$. Suppose $(X_A, \Delta_A)$ is a family of characteristic $p \gg 0$ models of $(X, \Delta)$. Then, there are positive constants $C_1$ and $C_2$ and a dense open set $U \subset \Spec(A)$ such that for any closed point $y \in U$, we have
\[\s(X_y, \Delta_y) > C_1, \]
and similarly,
\[ \FA(X_y, \Delta_y) > C_2 . \]
Note that the characteristic of the residue field (and hence the meaning of the $F$-signature and the $\FA$-invariant) depends on the closed point $y$.
\end{conj}

Now, we observe that in both \Cref{conj:alpha-lim-positive} and \Cref{strongerconjversion}, the version for the $F$-signature is equivalent to the version for the $\FA$-invariant. This follows immediately from the following lemma.

\begin{lem} \label{alphavss}
Let $(X, \Delta)$ be a log Fano pair over $\C$. Set $d = \dim (X)$ and $V = \vol(-K_X - \Delta)$. Then there exist positive real-valued functions $f$ and $g$ depending only on $d$ and $V$ such that
for any $p \gg 0$ and $(X_p, \Delta_p)$ a characteristic $p$-model of $(X, \Delta)$, we have
\begin{enumerate}
    \item If $\s (X_p, \Delta_p) \geq C$ then $\FA(X_p, \Delta_p) \geq f(C)$.
    \item And conversely, if $\FA(X_p, \Delta_p) \geq C$ then $\s (X_p, \Delta_p) \geq g(C)$.
\end{enumerate}
\end{lem}
\begin{proof}
For a log Fano pair $(X_p, \Delta_p)$ over a perfect field of characteristic $p>0$, let $\s_p$ and $\alpha_p$ denote the $F$-signature and the $\FA$-invariant of $(X_p, \Delta_p)$ respectively. Then, following \cite[Theorem 4.7]{PandeFrobeniusAlpha} we have:
$$ \frac{2 \, \alpha_p ^{d+1} \,  V}{(d+1)!} \leq \s_p \leq  \frac{2 \, \big( (\frac{1}{2})^{d+1} - (\frac{1}{2} - \alpha_p) ^{d+1} \big) \, V}{(d+1)!}. $$
Note that here we have used the fact that the dimension of $X_p$ is also $d$ and the volume of $-K_{X_p} - \Delta_p$ is $V$. Thus, we may take
\[ g(C) = \frac{2V C^{d+1}}{(d+1)!} \]
and 
\[ f(C) = \frac{1}{2} -  \left( \left(\frac{1}{2} \right)^{d+1} - \frac{C (d+1)!}{2V} \right )^{1/d+1} .\]
Note that the function $f(C)$ makes sense for all $ 0 \leq C \leq \frac{V}{2^d \, (d+1)!}$. This is sufficient because by \cite[Corollary 4.8]{PandeFrobeniusAlpha}, we always have $\s_p (X_p, \Delta_p) \leq \frac{V}{2^d (d+1)!}$. This completes the proof of the lemma.
\end{proof}

\begin{rem}
    These conjectures are known to hold in the following cases using a variety of techniques:
    \begin{enumerate}
        \item finite quotient singularities (with $\Delta = 0$) where the $F$-signature is essentially determined by the size of the finite group   \cite{HunekeLeuschkeTwoTheoremsAboutMaximal}.
        
        \item toric singularities (and $\Delta = 0$) using formulas for computing the $F$-signature in terms of the convex geometry of the defining cones  \cite{SinghFSignatureOfAffineSemigroup}, \cite{VonKorffthesis}.
        
        \item diagonal hypersurfaces using difficult computations in the Representation Ring of Han and Monsky \cite{CaminataShidelerTuckerZermanFsignatureofdiagonalhypersurfaces}. For the special case of quadric hypersurfaces, one can also study the decomposition of Frobenius push-forwards directly \cite{TrivediHKDensitfunctionofquadrics}. 
        \item full flag varieties using the $\FA$-invariant and the irreducibility of the Steinberg representations \cite{PandeFrobeniusAlpha}.
        \item Three-point configurations on $\PP^1$ i.e., $X = \PP^1$ and $\Delta $ has at most three components using explicit computation using Groebner bases \cite{BCPTlimitFsigarxiv}.
        \item Reductive quotient singularities using non-standard techniques and ultra-limits \cite{TakagiYamaguchiUniformPositivity}.
    \end{enumerate}
    In cases (1)-(5), it is additionally known that the limit of the $F$-signatures exist as $p \to \infty$.
\end{rem}

\subsection{Local to global reduction.}
Now we explain how \Cref{conj:alpha-lim-positive} and \Cref{strongerconjversion} for log Fano pairs imply the corresponding versions for arbitrary KLT singularities $(X,\Delta,x)$. For this, we recall the construction of degeneration to the orbifold cone over a Koll\'ar component:

\begin{defi} \label{dfn:Kollarcomponent}
    Let $x \in (X, \Delta)$ be a closed point, where $(X, \Delta)$ is a KLT pair over $\C$. Then, a proper birational map $\pi: Y \to X$ provides a Koll\'ar component $S$ of $(X, \Delta)$ if the following conditions hold:
    \begin{enumerate}
        \item $S \subset Y$ is the unique, prime $\pi$-exceptional divisor on $Y$ and $\pi(S) = \{x\}$. 
        \item $\pi$ is an isomorphism over $X \setminus \{x\}$.
        \item The pair $(Y, S + \Delta_Y)$ is PLT, where $\Delta_Y$ denotes the strict transform of $\Delta$ on $Y$ (\Cref{def:KLT}).
        \item $-S$ is $\Q$-Cartier and ample over $X$.
    \end{enumerate}
\end{defi}

Now assume that $X = \Spec(R)$ is an affine variety. Given a Koll\'ar component $S$ of $(X, \Delta)$, the pair $(S, \Delta_S)$ is a projective, log Fano pair of one dimension less than the dimension of $(X, \Delta)$. Here $\Delta_S$ is defined by the adjunction formula
\[ (K_Y + S + \Delta_Y)|_{S} = K_S + \Delta_S.  \]
Furthermore, let $v$ denote the valuation centered at $x$ defined by the prime divisor $S$, and let $\mathfrak{a}_m (v) = \{f \in R \, | \, v(f) \geq m \}$ be the $m^{\text{th}}$-valuation ideal. Following \cite[Section 2.4]{LiXuStabilityofvaluations} (also see \cite[Section 2.4]{LiuZhuangSharpnessofTianscriterion}), let $R_{*}$ denote the associated graded ring
\[ R_* = \bigoplus_{m \geq 0} \fa_m(v)/\fa_{m+1}(v). \]
Then, there is a natural effective $\Q$-divisor $\Delta_C$ on $C = \Spec(R_*)$ so that $(C, \Delta_C)$ is a KLT pair. In particular, $R_*$ is a Noetherian normal domain. Moreover, if $d$ is a positive integer such that $dS$ is Cartier and $-d(K_Y + S + \Delta_Y)$ is Cartier, then there is an associated degree $d$ finite map $h: C \to C_d = C (S; -d(K_S + \Delta_S))$ such that
\begin{equation} \label{eqn:orbifoldconepullback} h^* (K_{C_d} + \Delta_{C_d}) = K_C + \Delta _C. \end{equation}
Here, $C_d := C(S; -d(K_S + \Delta_S))$ denotes the affine cone over $S$ with respect to the ample divisor $-d(K_S + \Delta_S)$ and $\Delta_{C_d}$ denotes the cone over $\Delta_S$. Thus, we say that $(C, \Delta_C)$ is the \emph{orbifold cone} over the log Fano pair $(S, \Delta_S)$.

Again assuming $X = \Spec(R)$ is affine, any Koll\'ar component $S$ of the KLT pair $(X, \Delta, x)$ defines a degeneration of $(X, \Delta)$ to the orbifold cone $C = \Spec(R_*)$. Following \cite[Section 4.1]{LiXuStabilityofvaluations}, this is obtained by considering the extended Rees algebra $\cR$ which is the graded subring $\cR \subset R[t, t^{-1}]$ defined as
\[  \cR = \bigoplus _{k \in \Z}  \fa_{k} (v) t^{-k}. \]
Then let $\cX = \Spec(\cR)$ denote the total space and we have the following properties for $\cR$ and $\cX$:
\begin{enumerate}
    \item The map $\varphi: \cX \to \A^1$ corresponding to the inclusion $\C[t] \subset \cR$ is flat.
    \item We have $\cR \otimes _{\C[t]} \C[t,t^{-1}] \isom R[t, t^{-1}]$. In other words, the restriction $\cX|_{\A^1 \setminus \{0\}}$ is just the product $X \times (\A^1 \setminus \{0\}) $.
    \item $\cR/t\cR \isom R_*$, i.e., the Cartier divisor defined by $t$ is isomorphic to $C = \Spec(R_*)$.
    \item If $\fm$ denotes the maximal ideal of $R$ corresponding to $x$, the ideal
    \[ \mathfrak{p} = \bigoplus_{k>0} \fa_k (v) t^{-k} \oplus \bigoplus _{k \geq 0} \fm t^{k} \] defines a section $i : \A^1 \to \cX$ such that $\varphi \circ i = \text{id}_{\A^1}$. Thus, under this degeneration of $X$ to $C$, the point $x$ degenerates to the vertex of $C$ defined by the homogeneous maximal ideal of $R_*$.
    \item Let $K_\cX$ denote the canonical divisor of $\cX$ and $\Delta_{\cX}$ denote the closure of the pull-back of $\Delta$ in $\cX$. Then $K_\cX + \Delta_\cX$ is $\Q$-Cartier and we have
    \[ (K_\cX + \Delta_\cX)|_{C} = K_C + \Delta_C, \]
    where $\Delta_C$ is as defined above and satisfies \Cref{eqn:orbifoldconepullback}.
\end{enumerate}

Using this construction, we prove the following theorem:

\begin{thm} \label{thm:localtoglobalreduction}
  Let $(X, \Delta)$ be a KLT pair over $\C$, and $x \in X$ be a closed point. Suppose $(S, \Delta_S)$ is a Koll\'ar component of $(X, \Delta)$ (\Cref{dfn:Kollarcomponent}). Assume that \Cref{strongerconjversion} holds for the pair $(S, \Delta_S)$, then for any model $(X_A, \Delta_A, x_A)$ of $(X, \Delta, x)$ over a finitely generated smooth $\Z$-algebra $A$, there exists a dense open set $U \subset \Spec(A)$ and a positive constant $C>0$ such that we have
  \[ \s(X_y, \Delta_y, x_y) \geq C  \]
  for all closed points $y \in U$.
  \end{thm}
\begin{proof}
        By enlarging $A$ if necessary, we may assume that we also have a model for $Y, \pi, S, \Delta_S, C$ and $\Delta_C$ over $A$. Moreover, we may also assume that properties (1) to (5) listed above hold over each closed point of $\Spec(A).$

        In this situation, by our assumption there exists a constant $C>0$ and a dense open set $U \subset \Spec(A)$ such that for each closed point $y \in U$, we have
        $\s(S_y, \Delta_{S_y}) \geq C$. Therefore, to prove the theorem it is sufficient to justify that for each closed point $y \in U$, we have $\s(X_y,\Delta_y, x_y) \geq  A_{X, \Delta}(S)\ \s(S_y, \Delta_{S_y}) $.

        To see this, we first note that applying \cite[Corollary 1.2]{TaylorAdjunctionforFsig}, we have $\s(\cX_y, \Delta_{\cX_y}, \fn) \geq \s(C_y, \Delta_{C_y}, \fm_{0,y})$. Here $\fm_{0,y}$ denotes the closed point corresponding to the homogeneous maximal ideal of $R_{*, y}$, and $\fn$ corresponds to the same point in $\cX_y$. Then using \cite[Theorem 4.9]{PolstraTuckerCombinedApproach}, we have
        \[ \s(\cX_y, \Delta_{\cX_y}, \fn) \leq \s(\cX_y, \Delta_{\cX_y}, \mathfrak{p}_y) \leq \s(\cX_y|_{\A^1 \setminus \{0\}}, \Delta_{\cX_y}|_{\A^1 \setminus \{0\}}, (\mathfrak{p}_y, t^{-1}))\]
        where $\mathfrak{p}_y$ denotes the reduction of the ideal $\mathfrak{p} = \bigoplus_{k>0} \fa_k (v) t^{-k} \oplus \bigoplus _{k \geq 0} \fm t^{k} $. But we note that $\s(\cX_y|_{\A^1 \setminus \{0\}}, \Delta_{\cX_y}|_{\A^1 \setminus \{0\}}, (\mathfrak{p}_y, t^{-1}))$ is just the $F$-signature $\s(R_y \otimes _{k_y} k_y [t,t^{-1}], \Delta_{R_y}, \fm_{y}[t,t^{-1}]) $, which in turn is equal to $\s(R_y, \Delta_{R_y}, \fm_y)$ by Theorem 3.6 of \cite{CRSTBertiniTheorems}. Thus, we have shown that $\s(X_y, \Delta_{X_y}, x_y) \geq  \s(C_y, \Delta_{C_y}, \fm_{0,y})$.

        Now we will show that $\s(C_y, \Delta_{C_y}, \fm_{0,y})$ is equal to $ A_{(X, \Delta)} (S) \ \s(S_y, \Delta_{S_y})$ which will complete the proof. Note that by definition, $\s(S_y, \Delta_{S_y}) = r \, \s(\oplus_{j \geq 0} (H^ 0 (S, \cO_S (-rj (K_S + \Delta_S)))), \Delta_r)$ for any sufficiently divisible $r$. But at the same time,
        we have $ C= \Spec (R_*) \isom \Spec \oplus _{j \geq 0} H^0 (S, \cO (\lfloor -jS|_S \rfloor) ) $ by \cite[Proposition 2.10]{LiuZhuangSharpnessofTianscriterion} and we have
        \[ K_S + \Delta_S = (K_X + \Delta + S)|S \sim A_{X, \Delta} (S) S |S .\]
        Therefore, choosing a $d$ such that $r = \frac{d}{ A_{X, \Delta} (S)}$ is integral, we have the degree $r$ map $h_r: C \to C_r$ such that $h^* _ r (K_{C_{r}} + \Delta_{C_{r}}) = K_{C} + \Delta_{C}$. Then, applying \cite[Theorem 4.4]{CarvajalRojasSchwedeTuckerFundamentalGroups}, we have $\s(C, \Delta_{C}) = r \, \s(C_{r}, \Delta_{C_{r}}) = A_{X, \Delta} (S) \, \s (S, \Delta_S)$. This completes the proof of the theorem.
\end{proof}

In the above proof, we have shown the following useful statement that we record for future reference:

\begin{prop} \label{Prop:lowerboundKollarcomponent}
    Let $k$ be an algebraically closed field of characteristic $p$ and suppose that we have a normal Koll\'ar component $(S, \Delta_S)$ of a KLT singularity $x \in (X, \Delta)$. Then, we have
    \[  \s(X, \Delta, x) \geq {A_{X, \Delta} (S)} \, \s (S, \Delta_S).  \]
\end{prop}

\begin{rem}
    \Cref{thm:localtoglobalreduction} was also independently obtained by Takagi and Yamaguchi in \cite{TakagiYamaguchiUniformPositivity}.
\end{rem}

\subsection{Isotrivial degenerations}
We now prove that we may use isotrivial degenerations to prove \Cref{conj:alpha-lim-positive}.

\begin{defi} \label{isotrivialdfn}
    Let $X$ be a $\Q$-Fano variety over $\C$. Suppose $C$ is a smooth affine curve over $\C$ and let $c$ be a fixed point on $C$. Set $C^{\circ} =  C \setminus \{c\}$. Then an \emph{isotrivial degeneration} of $X$ over $(c \in C)$ to a normal projective variety $X_{0}$ is given by a flat, projective map $\pi: \cX \to C$, and a ($\pi$-)ample $\Q$-divisor $\cL$ on $\cX$ such that:
    
    \begin{enumerate}
        \item $\pi$ is trivial over $C^\circ$: 
        \[ \cX^{\circ} := \pi^{-1}(C^{\circ}) = \cX  \times_{C} C^{\circ} \isom X \times_{\C} C^{\circ} .\]
        \item the fiber over $c$ is isomorphic to $X_{0}$:
        \[ \cX_{c} := \cX \times_{C} \Spec(\kappa(c)) \isom X_0. \]
        \item The restriction (as a $\Q$-Cartier divisor) $\cL|_{\cX^{\circ}}$ to $\cX^{\circ}$ is linearly equivalent to $-K_{\cX^{\circ}}$, the anti-canonical divisor on $\cX^{\circ}$.
    \end{enumerate}  
\end{defi}

\begin{thm}\label{variationthm}
    Let $X$ be a $\Q$-Fano variety over $\C$. Assume that $X$ has an isotrivial degeneration to a $\Q$-Fano variety $X_{0}$. Then, if either \Cref{conj:alpha-lim-positive} or \Cref{strongerconjversion} is true for $X_0$, then the same is true for $X$. Moreover, we have $\s_{-} (X) \geq \s_{-} (X_0)$
    and 
    $ \alpha_{-} (X) \geq \alpha_{-} (X_0) $.
\end{thm}

We first note some technical properties of isotrivial degenerations to normal varieties. 
\begin{lem} \label{propertiesofisotriv}
Suppose we have an isotrivial degeneration $\pi: (\cX, \cL) \to C$ of a $\Q$-Fano variety $X$ to a normal projective variety $X_{0}$ over a smooth, affine pointed curve $(c \in C)$. Then,
\begin{enumerate}
    \item $\cX$ is a normal variety over $\C$.
    \item Both $\cX$ and $X_{0}$ are $\Q$-Gorenstein with ample the anti-canonical divisor. Moreover, we have $\cL \sim_\Q -K_{\cX|_C}$.
    \item For any $m \geq 0$, every divisor $\cD^{\circ}$ on $\cX^{\circ}$ that is linearly equivalent to $m\cL|_{\cX^{\circ}}$ extends uniquely to a $\Q$-Cartier divisor $\cD$ on $\cX$ such that $\cD$ is linearly equivalent to $m\cL$ and $\cD$ does not vanish along $X_0$.
    
\end{enumerate}
\end{lem}
\begin{proof}
First, we note that since $\cX^{\circ}$ is isomorphic to the product $X \times_{\C} C^{\circ}$, we may shrink $C$ to a neighbourhood of $c$ to prove the theorem. Hence, without loss of generality, we may assume that the point $c$ is a principal divisor on $C$. Therefore, $X_{0}$ is a principal divisor on $\cX$.
    \begin{enumerate}
        \item Since $X_{0}$ is a principal divisor on $\cX$ and $X_{0}$ is assumed to be normal it follows that $\cX$ is normal.
        \item Since $\cX$ is normal, let $K_{\cX}$ denote a canonical Weil-divisor on $\cX$. We claim that $\cL \sim_{\Q} -K_{\cX}$. Since by assumption, $\cL|_{\cX^{\circ}} \sim_{\Q} -K_{\cX^{\circ}}$, taking closures in $\cX$ (as $\Q$-Weil divisors), we must have
        $$ -K_{\cX} \sim_{\Q} \cL + a X_{0}, \quad \text{for some $a \in \Q$} $$
        since $X_{0}$ is a prime divisor. But the claim follows from the fact that $X_{0}$ is actually principal. This shows that $\cX$ is $\Q$-Gorenstein and $-K_{\cX}$ is ample. \\
        Next, let $U \subset X_{0}$ denote the smooth locus of $X_{0}$, and $V \subset 
 \cX_{\text{sm}}$ be an open set such that $V|_{X_{0}} = U$. Note that $U$ is a principal divisor on $V$. Then, applying the adjunction formula, we have
 $$ K_U \sim K_V + U |_{U} \sim K_{V}|_U. $$
 Taking closures and using the normality of $X_{0}$ and $\cX$, we get that
 $$ -K_{X_{0}} \sim - K_{\cX}|_{X_{0}} \sim_{\Q} \cL|_{X_{0}} . $$
 This proves that $X_{0}$ is $\Q$-Gorenstein as well, and since $\cL$ is ample on $\cX$, $-K_{X_{0}}$ is ample on $X_{0}$.
 \item Given a divisor $\cD^{\circ}$ on $\cX^{\circ}$, we may take its closure $\cD$, which is a Weil-divisor on $\cX$. Since $\cD|_{\cX^{\circ}} \sim \cL|_{\cX^{\circ}}$, we must again have  
    $$ \cD \sim_{\Q} \cL + a X_{0}, \quad \text{for some $a \in \Q$}. $$
    Now the claim follows from the fact that $X_{0}$ is a principal divisor. Note that this automatically makes $\cD$ a $\Q$-Cartier divisor. \qedhere
    \end{enumerate}
\end{proof}

\begin{proof}[Proof of \Cref{variationthm}]
Given an isotrivial degeneration $\pi: \cX \to C$ of $X$ to $X_0$ over a pointed curve $(c \in C)$, we may choose models for $\cX$, $X_0$, $\pi$, $c$ and $C$ over a smooth $\Z$-algebra $A$. Furthermore, we may assume that the isomorphisms in the properties (1) - (3) in \Cref{isotrivialdfn} hold over $A$ and hence hold over every closed fiber over $\Spec(A)$. Thus, for any $p \gg 0$, we may assume that we have an isotrivial degeneration of the characteristic $p$ models of $X$ to the characteristic $p $ reductions of $X_0$. The rest of the proof is exactly the same as \Cref{thm:localtoglobalreduction} and relies on the semicontinuity properties of the $F$-signature and the $\FA$-invariant \cite[Theorem~5.8]{PandeFrobeniusAlpha}. So we omit the details here.
\end{proof}

\subsection{Spherical Fano varieties and normal toric degenerations.}

By applying \Cref{variationthm}, we can use isotrivial degenerations to normal toric varieties to prove \Cref{strongerconjversion} and \Cref{conj:alpha-lim-positive} for spherical Fano varieties.

\begin{thm} \label{normaltoricdegeneration}
      Let $X$ be a $\Q$-Fano variety over $\C$. Assume that $X$ admits an isotrivial degeneration (over some smooth curve $(c \in C)$) to $X_0$ where $X_0$ is a normal (projective) toric variety. Then, \Cref{conj:alpha-lim-positive} and \Cref{strongerconjversion} hold for $X$. Moreover, we have
    \[ \alpha_{-} (X) \geq \alpha_\C (X_0) \]
    where $\alpha_\C$ denotes Tian's $\alpha$-invariant over $\C$.
\end{thm}
\begin{proof}
    Since $X_0$ is a normal projective toric variety, it is of log Fano type. By \Cref{propertiesofisotriv}, we have that $-K_{X_0}$ is automatically ample and $\Q$-Cartier. Moreover, suppose $X_0$ is a projective toric variety corresponding to a fan $\cF$. Then, we may pick characteristic $p \gg 0$-models of $X_0$ that are also normal toric varieties corresponding to $\cF$. In particular, \Cref{strongerconjversion} holds for $X_0$ and the $\FA$-invariant and the $F$-signature of all characteristic $p$ models are independent of $p$. Furthermore, by \cite[Theorem~4.12]{PandeFrobeniusAlpha}, we have that $\FA(X_{0,p}) = \alpha_\C (X_0)$ for each characteristic $p$-model of $X_0$.
    Therefore, we may apply \Cref{variationthm} to conclude that \[ \alpha_{-} (X) \geq \alpha_{\C} (X_0) . \qedhere \]
\end{proof}

\begin{defi} \label{dfn:spherical}
    Let $X$ be a normal projective variety over $\C$. Then $X$ is called \emph{spherical} if there exists a reductive algebraic group $G$ acting on $X$ and a Borel subgroup $B$ of $G$ such that $B$ has a dense orbit in $X$. A \emph{spherical $\Q$-Fano} variety is a $\Q$-Fano spherical variety where the action of $G$ is linearized to the ample line bundle defined by $-rK_X$ for some $r >0$.
\end{defi}

\begin{cor} \label{cor:spherical}
Let $X$ be a spherical $\Q$-Fano variety over $\C$.
 Then, \Cref{conj:alpha-lim-positive} and \Cref{strongerconjversion} hold for $X$.
\end{cor} 
\begin{proof}
    By \cite[Theorem 3.2]{AlexeevBrionToricDegenerationsofspherical}, every spherical $\Q$-Fano variety isotrivially degenerates to a toric variety $X_0$. Furthermore, by \cite[Proposition 2.4 and Remark 3.9]{AlexeevBrionToricDegenerationsofspherical}, we have that $X_0$ is a normal toric variety. The corollary now follows from \Cref{normaltoricdegeneration} and \Cref{variationthm}.
\end{proof}
%\begin{cor}
%Let $X$ be either:
%\begin{enumerate}
%    \item a smooth cubic hypersurface in $\PP^n$ for $n \geq 3$.
%    \item a smooth complete intersection of two quadrics in $\PP^n$ for $n \geq 4$.
%\end{enumerate}
%Then, $s_{-}(X) > 0$ and $\alpha_{-}(X) > 0$.
%\end{cor}

\section{Toric degeneration of hypersurfaces of very low degree} \label{section:hypersurfaces}

In this section, we work over $\bC$. For notions regarding K-stability and K-moduli spaces, we refer to \cite{XuKstabilitybook} for the precise definitions.

Let $X\subset \bP^{n}$ be an arbitrary smooth hypersurface of degree $d\geq 3$. Let $\cM_{k,d} := \bP(H^0(\bP^k, \cO(d)))\sslash \PGL_{k+1}$ denote the GIT moduli space of degree $d$ hypersurfaces of $\PP^k$. Consider the rational map  $\Phi:\bG(k,n)\dashrightarrow \cM_{k,d}$ by taking linear sections of $X$ with a $k$-plane $\bP^k$ in $\bP^n$. Then, Starr in \cite{Sta06} showed that the map $\Phi$ is dominant if $n\geq \binom{d+k -1}{k} +k-1$. We consider the special case where $k=d$, and we set $n_0=n_0(d):=\binom{2d -1}{d} +d-1$.

\begin{thm} \label{thmlowdegree}
Assume that $n\geq n_0$. Then there exists an isotrivial degeneration $X\rightsquigarrow X_0$ (\Cref{isotrivialdfn}) such that $X_0$ is a toric hypersurface of degree $d$ in $\bP^n$ with Gorenstein canonical singularities. Moreover, $X_0$ has the equation $x_0^d= x_1\cdots x_d$.
\end{thm}

Combining this with \Cref{variationthm}, we obtain:

\begin{cor} \label{cor:hypersurface}
Let $X \subset \P^n$ be a smooth hypersurface of degree $d$ over $\C$. Assume that $n \geq \binom{2d -1}{d} +d-1 $. Then \Cref{strongerconjversion} holds for $X$.    
\end{cor}

\begin{lem} \label{lem:embeddedisotriv}
    Suppose $X_1 $ and $X_2$ are two projective subvarieties of $\P^n$ with the same Hilbert polynomial. Then, there is an \emph{embedded} isotrivial degeneration of $X_1 \rightsquigarrow X_2$ if and only if the orbit closure $\overline{\PGL_{n+1} \cdot X_1}$ of $X_1$ intersects the $\PGL_{n+1}$-orbit of $X_2$ in the corresponding Hilbert scheme.
\end{lem}
Here, an embedded isotrivial degeneration is a family $\cX \to C$ as in \Cref{isotrivialdfn} that arises from a subscheme of $\PP^n \times C \to C$.
\begin{proof}
    The proof relies on the existence and projectivity of the Hilbert scheme parametrizing subschemes of $\PP^n$ with a fixed Hilbert polynomial. Let $P$ be the Hilbert polynomial of $X_1$ and let $Z : = \overline{\PGL_{n+1} \cdot [X_1]} \subset \Hilb_{\P^n} (P)$ denote the orbit closure of $X_1$ under the natural $\PGL_{n+1}$-action on $\Hilb_{\P^n} (P)$. First, suppose there is a flat subscheme $S \subset \P^n \times C $ over some pointed curve $(0 \in C)$ that defines an isotrivial degeneration of $X_1 \rightsquigarrow X_2$. Then, by considering the corresponding map $ \pi: C \to \Hilb_{\P^n} (P)$, we note that the image $\pi (C \setminus \{0\})$ is contained in the orbit of $X_1$, while $\pi(0) $ is contained in the orbit of $X_2$. Since $\Hilb_{\P^n}(P)$ is proper and the orbit closure $Z$ is a closed subscheme of $\Hilb_{\P^n}(P)$, $\pi(0)$ must also be contained in the subscheme $S$. Therefore, $S$ intersects the orbit of $X_2$.

    Conversely, suppose $z \in Z \cap (\PGL_{n+1} \cdot X_2)$. Then, since $Z$ is a projective scheme, there is a map from a smooth curve $\pi: C \to Z$ that joins $z$ to some point of $\PGL_{n+1} \cdot X_1$ in $Z$. Let $0  \in C$ be the point mapping to $z$. Since the orbit of $X_1$ is dense in $Z$, there must be an open set $U \subset C$ that maps into the orbit of $X_1$. Now, since $C \setminus U$ is finite, we may replace $C$ by $U \cup \{0\}$ so that $\pi: C \to Z$ now defines an embedded isotrivial degeneration of $X_1$ to $X_2$.
\end{proof}

\begin{lem} \label{lem:compositionofisotrivdegn}
    For subschemes $X_1, X_2, X_3$ of $\PP^n$ with the same Hilbert polynomial, if $X_1$ admits an embedded isotrivial degeneration to $X_2 $ (as in \Cref{lem:embeddedisotriv}) and if $X_2$ admits an embedded isotrivial degeneration to $X_3$, then $X_1$ admits an embedded isotrivial degeneration to $X_3$.
\end{lem}
\begin{proof}
    By \Cref{lem:embeddedisotriv}, it is sufficient to show that the closure $Z$ of the $\PGL_{n+1}$-orbit of $X_1$ intersects the orbit of $X_3$ in the corresponding Hilbert scheme. Note that $Z$ is a $\PGL_{n+1}$-stable closed, projective subscheme of the Hilbert scheme. Therefore, since $Z$ intersects the orbit of $X_2$, it must contain the closure of the orbit of $X_2$ and hence intersects the orbit of $X_3$ as well.
\end{proof}
\begin{lem}\label{lem1}
Let $V:= (x_0^d = x_1\cdots x_d)\subset \bP^d$. Then $V$ is K-polystable.
\end{lem}

\begin{proof}
Clearly $V$ admits an action of the reductive group $G:=\bT\rtimes \fS_d$ where $\bT=\bG_m^{d-1}$ acts diagonally and $\fS_d$ permutes the coordinates $x_1,\cdots, x_d$. It is not hard to check that there does not exist a $G$-invariant irreducible closed subvariety of $V$. Hence by \cite{Zhu21} we know that $V$ is K-polystable.
\end{proof}

\begin{proof}[Proof of \Cref{thmlowdegree}]
    By \Cref{lem:embeddedisotriv}, it is enough to show that the closure of the $\PGL_{n+1}$-orbit of $X$ intersects the $\PGL_{n+1}$-orbit of $X_0$ in the projective space parametrizing degree $d$ hypersurfaces in $\PP^n$. First, we note that by deformation to the normal cone, any given hypersurface $X_1$ in $\PP^n$ isotrivially degenerates to the cone over a hyperplane section of $X_1$. Applying this observation and \Cref{lem:compositionofisotrivdegn} iteratively, we see that if $Z$ is a linear section of $X_1$, then $X_1$ admits an embedded isotrivial degeneration to the cone over $Z$ in $\PP^n$.

   Next, by Lemma \ref{lem1} we know that $V_0 = (x_0^d = x_1\cdots x_d)\subset \bP^d$ is K-polystable which implies that it is GIT polystable by Paul-Tian \cite{PT06}. Let $\Phi:\bG(d,n)\dashrightarrow \cM_{d,d}  := \bP(H^0(\bP^d, \cO(d)))\sslash \PGL_{d+1}$ be the rational map obtained by intersecting a $d$-plane in $\P^n $ with $X$. 
   
   Now, since the map $\Phi:\bG(d,n)\dashrightarrow \cM_{d,d}$ is dominant under our assumptions on $n$ and $d$ by \cite{Sta06}, we consider a map from a pointed curve $\pi: (0 \in C) \to \bP(H^0(\bP^d,  \cO(d)))\sslash \PGL_{d+1}$ such that for each $c \neq 0$, $\pi(c)$ is a GIT-stable hypersurface in the image of $
   \Phi$ and $\pi(0)$ is equal to $[V_0]$. After replacing $C$ by a finite cover, we may lift $\pi$ to a map $\tilde{\pi}: C \setminus \{0\} \to \bP(H^0(\bP^d,  \cO(d)))$, so that we have a family $\psi: \cV \to C \setminus \{0\} $ of hypersurfaces in $\P ^d$ such that the fiber $\psi^{-1}(c)$ is a GIT-stable hypersurface isomorphic to $\pi(c)$ for each $c \neq 0$. By the properness of the GIT-moduli space of hypersurfaces, we know that $\psi$ can be completed to a family $\overline{\psi}: \overline{\cV} \to C $ so that $\cV_0$, the fiber of $\overline{\psi}$ over $0 \in C $, is a GIT-semistable hypersurface whose polystable degeneration is isomorphic to $V_0$.

   Next, since the general fiber of $\overline{\psi}$ is in the image of $\Phi$, by shrinking $C$ if required, we may trivialize the Grassmann-bundle locally over $\Phi^{-1} (\pi(C))$. Thus, we may assume that there is a trivial family of $d$-planes: $ \cH := C  \times \P^d \subset C \times \P^n$ such that for each $c \neq 0$, by intersecting $\cH_c$ with $X$ we have $ \cH_{c} \cap X = \cV_c $. Viewing the family $\overline{\psi}$ inside $\overline{\cV} \subset C \times \P^d \subset C \times \P^n $, we may take the projective cones over $\cV$. Thus, we get a family of hypersurfaces $\cY \subset C \times \PP^n  \to C$ such that for each $c$ the fiber $\cY_c$ is isomorphic to the projective cone over $\cV_c$.
   
    Lastly, now consider the orbit closure $\overline{\cO}$ of $X$ in the projective space of degree $d$ hypersurfaces in $\PP^n$. Since the orbit of $\cY_c$ is contained in $\overline{\cO}$ for each $c \neq 0$ (since $X$ isotrivially degenerates to the cone $\cY_c$), by the properness of $\overline{\cO}$, the orbit of $\cY_0$ must also be contained in $\overline{\cO}$ and therefore, $X$ admits an isotrivial degeneration to the projective cone over $\cV_0$ (by \Cref{lem:embeddedisotriv}). Since $V_0$ is the (GIT)-polystable degeneration of $\cV_0$, the projective cone over $V_0$ isotrivially degenerates to the projective cone over $V$. Therefore, by \Cref{lem:compositionofisotrivdegn} $X$ isotrivially degenerates to the projective cone over $V$.
\end{proof}

We also give another proof of \Cref{thmlowdegree} using K-moduli spaces instead of the GIT moduli space.

\begin{proof}[Second proof of \Cref{thmlowdegree}]
Firstly, by Lemma \ref{lem1} we know that $V= (x_0^d = x_1\cdots x_d)\subset \bP^d$ is K-polystable which implies that it is GIT polystable by Paul-Tian \cite{PT06}. 
Hence from the result of Starr \cite{Sta06}, there exists a map $\psi$ from a pointed curve $0\in T$ to $\cM_{d,d}$ such that $\psi(0)= [V]$ and $\psi(T\setminus \{0\})$ lies in the image of $\Phi$ intersecting the K-stable locus, i.e.  $[V_t]:=\psi(t)$ is smooth and K-stable for every $t\in T\setminus \{0\}$. After replacing $T$ by a quasi-finite base change, we can assume that $\psi|_{T\setminus \{0\}}$ lifts to a map $\tpsi:T\setminus \{0\}\to \bG(d,n)$. After a further base change, we may assume that there are maps $h_i: T\setminus \{0\} \to \bG(n-1, n)=(\bP^n)^\vee$ for $1\leq i\leq n-d$ such that $V_t = D_{1,t} \cap \cdots \cap D_{n-d,t}$ where $D_{i,t} = h_i(t)\cap X$ is a hyperplane section.

Below, for any pointed curve $( 0 \in C)$, we denote by $C^\circ$ the curve $C \setminus \{0\}$.

The maps $h_i$ give us families of hyperplanes $\cH_i \subset  \P^n \times T^{\circ} \to T^{\circ}$ that we may equivalently consider as a family of log-pairs  $\Phi: (\P^n _{T^\circ}, (1 - \frac{1}{d}) \sum _i \cH_i)\to T^{\circ} $. Note that by construction, the divisors $\cH_i$ intersect transversely with $X_{T^\circ} := X \times T^{\circ} \subset \P^n _{T^\circ}$. Thus, by restricting to $X_{T^\circ}$, we get a family of log Fano pairs $\phi_1: (X_{T^\circ}, \cD = (1 - \frac{1}{d}) \sum _i \cH_i |_{T^\circ} )\to T^{\circ} $. Since $V_t$ is K-stable for every $t\in T^{\circ}$, by Lemma \ref{lemmaconepolystable} we know that $(X, (1-\frac{1}{d})\sum_{i} D_{i,t})$ is K-semistable. Thus, $\phi_1$ is family of K-semistable log Fano pairs over $T^{\circ}$. Then, by the properness of the K-moduli space of  log Fano pairs (see \cite[Corollary~7.4]{BlumHalpernLeistnerLiXuProperness} and \cite{LiuXuZhuangHigherrankfinitegeneration}), upto a finite dominant map $(0 \in T_1) \to (0 \in T)$ from a pointed smooth curve $0 \in T_1$, we may extend the family $\phi_1$ (originally over $T^\circ$) to a family $\overline{\phi_1}: (\cX_1, \overline{\cD_1}) \to T_1$ of K-semistable log Fano pairs, such that
\[ \overline {\phi_1} |_{T_1 ^{\circ}} = \phi_1 \times _{T^\circ} T^\circ _1. \]
 Moreover, the $\Q$-divisor $\overline{\cD_1}$ will have coefficients in the set $\frac{1}{d} \N_{\geq 0}$. Let $(X_1, D_1)$ denote the fiber of $\overline{\phi_1}$ over $0$.

On the other hand, as in \Cref{lemmaconepolystable}, we can construct a K-polystable degeneration of the geometric generic fiber of $\phi_1$, and up to another finite dominant map $(0 \in T_2 ) \to (0 \in T_1)$, this extends to a family $\phi_2: \cX_2 ^{\circ} \to T_2 ^\circ $ such that the fibers of $\phi_2$ are the K-polystable degenerations of the fibers of $\phi_1 \times _{T_1 \circ} T_2 ^\circ$. Now, by the $\Theta$-reductivity of the K-moduli stack of log Fano pairs (precisely, using \cite[Theorem~5.2]{ABHLXReductivity}), there exists an extension $\overline{\phi_2} : (\cX_2, \cD_2) \to T_2$ of $\phi_2$ to a family over $T_2$ such that the fiber over $0 \in T_2$ (denoted by $(X_2 ,D_2) $) is K-semistable and such that there is a $\G_m$-equivariant degeneration of $(X_1, D_1)$ over $\A^1$ to $(X_2, D_2)$.  

Finally, we can consider the family of projective cones over $\psi$ (after base change to $T_2$), so that we have a family
$\overline{\phi_3} : (\cX_3, \cD_3) \to T_2 $ such that the fiber over $t \in T_2$ is the pair $(C(V_t), (1 -\frac{1}{d}) \sum _{i = 1} ^{n-d} (x_{d+i} = 0))$, where $C(V_t)$ denotes the projective cone (in $\P^n$) over $V_t$. Let $X_0$ denote the cone over $V_0$ in $\P^n$ (therefore, $X_0$ is defined by the equation $(x_0 ^d = x_1 \dots x_d)$). Since  by \Cref{lemmaconepolystable} we know that the K-polystable degeneration of $(X_t, (1-\frac{1}{d})\sum_{i} D_{i,t})$ is $(C(V_t), (1-\frac{1}{d})\sum_{i=1}^{n-d} (x_{d+i}=0))$, by the uniqueness of K-polystable degenerations (\cite[Theorem~1.1]{BlumXuUniquenessofpolystabledegenerations}), we get that the cone $(X_0,  (1 -\frac{1}{d}) \sum _{i = 1} ^{n-d} (x_{d+i} = 0))$ is the unique polystable limit of $(C(V_t), (1-\frac{1}{d})\sum_{i=1}^{n-d} (x_{d+i}=0)) $ as $t \to 0$. Therefore, by \cite[Theorem~3.2 (a)]{LiWangXuAlgebraicityofmetrictangentcones}, there is a $\G_m$-equivariant degeneration of $(X_2, D_2)$ over $\A^1$ (i.e., by a special test configuration) to $(X_0 ,   (1-\frac{1}{d})\sum_{i=1}^{n-d} (x_{d+i}=0)) $.

To conclude the proof, we note that the family $\overline{\phi_1}$ restricts to an isotrivial degeneration of $(X, -K_X)$ to $(X_1, -K_{X_1})$. See \Cref{isotrivialdfn} and \Cref{propertiesofisotriv}. Note that by \Cref{propertiesofisotriv} Part (b), $X_1$ is automatically $\Q$-Fano and we have
\[ -K_{X_1} \sim _\Q -c(K_{X_1} + D_1)\]
where $c = \frac{d(n-d+1)}{d}$, since the same is true for $X$. Similarly, the isotrivial degenerations that we have already constructed of $(X_1, D_1)$ to $(X_2, D_2)$ and of $(X_2, D_2)$ to $(X_0, (1-\frac{1}{d})\sum_{i=1}^{n-d} (x_{d+i}=0)) $ restrict to isotrivial degenerations of $(X_1, -K_{X_1})$ to $(X_2 -K_{X_2})$, and of $(X_2, -K_{X_2})$ to $(X_0, -K_{X_0})$. Thus, by Lemma \ref{lem:compositionofisotrivdegn}, $(X, -K_X)$ admits an isotrivial degeneration to $(X_0, -K_{X_0})$.
\end{proof}

\begin{lem}\label{lemmaconepolystable}
Suppose $X$ is a degree $d$ hypersurface in $\bP^n$ with $n\geq d$. Let $V$ be a linear section of $X$ by a $\bP^d$ which is a K-polystable hypersurface in $\bP^d$. Let $(H_i)_{1\leq i\leq n-d}$ be hyperplane sections of $X$ such that $V = \cap_i H_i$. Then $(X, (1-\frac{1}{d})\sum_i H_i)$ is K-semistable. Moreover, its K-polystable degeneration is $(X_0, (1-\frac{1}{d})\sum_{i=1}^{n-d} (x_{d+i}=0))$ where $X_0$ is the cone over $V$, i.e. if $V=(f(x_0,\cdots, x_d)=0)\subset\bP^d$ then $X_0=(f(x_0,\cdots, x_d)=0)\subset\bP^n$.
\end{lem}

\begin{proof}
For each fixed $d$, we do induction on $n\geq d$. Note that the base case  $n = d$  is trivial. Assume that the statement is true for $n$ and $d$. Now let $X'\subset \bP^{n+1}$, $(H_i')_{1\leq i\leq n+1-d}$ be hyperplane sections of $X'$ such that $V = \cap_i H_i'$. Let $X_0':= (f(x_0,\cdots, x_d) = 0) \subset \bP^{n+1}$. In a suitable projective coordinates we may assume that $H_i' = (x_{d+i} = 0)\subset \bP^{n+1}$. Let $X := X'\cap H_{n+1-d}'$, $H_i:= H_i'\cap H_{n+1-d}'$ for $1\leq i\leq n-d$, and $X_0':= X_0\cap (x_{n+1} = 0)$. By the induction hypothesis, we have that $(X, (1-\frac{1}{d})\sum_{i=1}^{n-d} H_i)$ (resp. $(X_0, (1-\frac{1}{d})\sum_{i=1}^{n-d} (x_{d+i}=0))$) is K-semistable (resp. K-polystable). By degeneration to normal cone, we know that $(X', (1-\frac{1}{d})\sum_{i=1}^{n+1-d} H_i')$ specially degenerates to $(C(X), (1-\frac{1}{d})\sum_{i=1}^{n-d} C(H_i) + (1-\frac{1}{d})(x_{n+1} = 0))$ which is K-semistable by \cite[Proposition 5.3]{LiXuStabilityofvaluations}. Thus by the openness of K-semistability \cite{BLX22, Xu20} we know that  $(X', (1-\frac{1}{d})\sum_{i=1}^{n+1-d} H_i')$ is K-semistable. Since $X_0'=C(X_0)$, the K-polystability of $(X_0', (1-\frac{1}{d})\sum_{i=1}^{n+1-d} (x_{d+i}=0))$ follows from the K-polystability generalization of \cite[Proposition 5.3]{LiXuStabilityofvaluations} (see \cite[Proposition 3.5]{LL19} through the YTD correspondence \cite{LiuXuZhuangHigherrankfinitegeneration} or \cite[Proposition 2.11]{LiuZhuangSharpnessofTianscriterion}).
%First of all, by iterated degeneration to normal cone we know that $(X, (1-\frac{1}{d})\sum_i H_i)$ admits a special degeneration to $(X_0, (1-\frac{1}{d})\sum_{i=1}^{n-d} (x_{d+i}=0))$
%Take iterated cone construction from \cite{LL19} and use degeneration techniques.\YL{add details}
\end{proof}

\begin{rem}
We note that Ilten and Lautsch \cite{IL22} prove that if $d\geq 2n-1$ then a general hypersurface $X$ admits a degeneration via a test configuration to a toric hypersurface. However, their construction of toric degeneration is in general not normal (hence not KLT).
\end{rem}

\section{Non-weakly exceptional singularities}
\label{section:Fregularity}

\subsection{Special divisors over non-weakly exceptional log del Pezzo surfaces.}

First, we will study non-weakly exceptional log del Pezzo surfaces over $k$.

\begin{defi} \label{dfn:weaklyexcepFano}
        A log Fano pair $(X, \Delta)$ over $k$ is weakly exceptional if for every effective $\Q$-divisor $\Gamma$ such that $\Gamma \sim _\Q -(K_X + \Delta)$, the pair $(X, \Delta+ \Gamma)$ is log canonical. In other words, we have $\alpha(X,\Delta) \geq 1$.
\end{defi}

The following theorem on extracting special divisors over surfaces in characteristic $p>0$ was obtained by Harold Blum and Joe Waldron and we thank them for pointing us to its proof. It is an analog of a theorem due to Zhuang in the complex case, and we include a proof for completeness.

\begin{thm} \label{thm:specialdivisor}
    Let $(S,\Delta)$ be a log del Pezzo surface over $k$ and assume that $(S, \Delta)$ is not weakly exceptional (i.e., $\alpha (S, \Delta)< 1$). Then, there exists a projective birational map $\mu: Z \to S$ and an effective $\Q$-divisor $D^*$ on $S$ such that the following conditions hold:
    \begin{enumerate}
        \item There is a unique prime divisor $E \subset Z$ and $-E$ is $\mu$-ample.

        \item We have $D^* \sim _\Q -(K_S + \Delta)$ and $(S, \Delta+ D^*)$ is log canonical.

        \item $E$ is the unique log canonical place of $(S, \Delta+ D^*)$, so that the pair $(Z, E + \Delta_Z + D^*_Z)$ is PLT (where $ \Delta_Z, D_Z ^*$ denote the strict transforms of $\Delta,D^*$ respectively).

        \item The divisor $-K_Z - \Delta_Z - E$ is big.

        \item The surface $Z$ is of Fano type, i.e., there exists a $\Q$-divisor $\Delta' \geq 0$ on $Z$ such that $-K_Z - \Delta'$ is ample and $(Z, \Delta')$ is KLT.
    \end{enumerate}
\end{thm}

\begin{proof}
    Since $(S, \Delta)$ is not weakly exceptional, we can find an effective divisor $D \sim_\Q -(K_S + \Delta)$ such that $(S, \Delta + tD)$ is log canonical but not KLT for some $t \in (0,1)$. Let $E$ denote an LC-place of $(S, \Delta + tD)$. Consider a projective log resolution  $h: W \to S$ of $(S, \Delta + D)$ containing $E$ and pick $0 < \delta \ll 1$ and write 
    \begin{equation} \label{eqn:logrespullback} K_W + \Delta_W + (t - \delta) D_W + G_1 = h^* (K_S + \Delta + (t- \delta)D) + G_2 , \end{equation}
    where $\Delta_W, D_W$ denote the strict transforms of $\Delta, D$ respectively and $G_1 $ and $G_2$ are effective exceptional divisors such that: the support of $G_1$ contains all divisors with log discrepancy over $(S, \Delta + tD)$ less than $1$ and $G_2$ contains all divisors with log discrepancy larger than $1$. We may assume that $E$ appears in $G_1$ (since $E$ is an LC-place) and that every exceptional divisor of $h$ appears in either $G_1$ or $G_2$.  

    Let $F$ denote the reduced divisor containing the support of $G_1 \setminus E$. It follows from \Cref{eqn:logrespullback} that if we pick $0 < \varepsilon \ll \delta$ then, $(W, \Delta_W + (t - \delta)D_W + G_1 + \varepsilon F) $ is KLT and moreover, since $\Delta_W + (t - \delta)D_W + G_1 + \varepsilon F $ is effective, it is big over $S$ (since $h$ is a projective, birational map). Therefore, by \cite[Theorem 6.5]{TanakaMMPforlogsurfaces} we may run a $K_W + \Delta_W + (t - \delta)D_W + G_1 + \varepsilon F$- MMP over $S$ which ends in a $\Q$-factorial terminal model $W \overset{g}{\to} Z \overset{\mu} {\to} S$ such that $K_Z + g_* (\Delta_W + (t - \delta)D_W + G_1 + \varepsilon F)$ is nef over $S$. Using \Cref{eqn:logrespullback} again, we have
    \[K_Z + g_* (\Delta_W + (t - \delta)D_W + G_1 + \varepsilon F) = \mu ^*(K_S + \Delta + (t - \delta)D) + g_* (G_2 + \varepsilon F)  .  \]
    Since $G_2$ and $F$ are effective divisors, by the negativity lemma, we must have that $g_* (G_2 + \varepsilon F) = 0$, so that $g$ contracts all $h$-exceptional divisors except $E$. Since we know that $E$ is an LC-place of $(S, \Delta + tD)$, we have
    \begin{equation} \label{eqn:specialdivisorpullback} K_Z + \Delta_Z + tD_Z + E = \mu^* (K_S + \Delta + tD),  \end{equation}
    and hence we know that $-E^2 > 0$, and hence $-E$ is ample over $S$. This shows part (1). Moreover, since $- (K_Z + \Delta_Z + E) \sim_\Q -(1- t)\mu^*(K_S + \Delta) + t D_Z $, 
    we see that $- (K_Z + \Delta_Z + E)$ is big (since $-(K_S + \Delta_S)$ was big and $t \in (0,1)$), which proves part (4).

    Next, set $m : = \text{ord}_E (D)$ and note that for $0 < \varepsilon \ll 1$, by \Cref{eqn:specialdivisorpullback} we have $K_Z + \Delta_Z + (t - \varepsilon) D_Z + ( 1 - m \varepsilon )E = \mu^* (K_S + \Delta + (t - \varepsilon)D)$. Therefore, $(Z, \Delta_Z +  (t - \varepsilon) D_Z + (1 - m\varepsilon) E)$ is KLT since $(S, \Delta + (t - \varepsilon)D)$ is.
     Moreover, $- (K_Z + \Delta_Z + (t - \varepsilon) D_Z + ( 1 - m \varepsilon )E) \sim_\Q -(1 - t + \varepsilon) \mu^* (K_S + \Delta)$ is big and nef. Therefore, since $-E$ is $\Q$-Cartier and $\mu$-ample, we see that $- (K_Z + \Delta_Z + (t - \varepsilon) D_Z + ( 1 - m \varepsilon  - \delta)E)$ is ample for $0 < \delta \ll 1$, which shows that $Z$ is of Fano type (Part (5)).

    Next, we run a $(-K_Z-\Delta_Z-E)$-MMP which terminates to $Z \to Z^{\rm m}$ such that, if $E^{\rm m} $, $\Delta_{Y^{\rm m}}$ denotes the strict transform of $E$ and $\Delta_Z$ on $Z^{\rm m}$, we have $-K_{Z^{\rm m}} -\Delta_{Z^{\rm m}} -E^{\rm m}$ is big and nef. We can run such an MMP because $Z$ is a $\Q$-factorial surface by \cite[Theorem 3.27]{TanakaMMPforlogsurfaces}. Moreover, by \cite[Proposition 3.29]{TanakaMMPforlogsurfaces}, $-K_{Z^{\rm m}} -\Delta_{Z^{\rm m}} -E^{\rm m}$ is big and semi-ample, and hence admits a birational ample model $\pi: Z \to Z^{\rm m} \to Z'$. Moreover, we see that $E$ is not contracted by $\pi$: first note that we may pick a rational number $0 < \lambda \ll 1$ and a general divisor $H \in |- (K_S + \Delta + (t - \lambda)D)|_\Q $ that avoids the singular points of $(S, \Delta + tD)$ and the center of $E$, so that $(S, \Delta + (t- \lambda)D + H)$ is KLT. Moreover, since $K_S + \Delta + (t - \lambda) D + H \sim _\Q 0$, we see that the crepant pull-backs of $(S, \Delta + (t- \lambda)D + H)$ to $Z$, $Z^{\rm m}$ and $Z'$ are all KLT as well. In particular, we see that $(Z, \Delta_Z)$, $(Z^{\fm}, \Delta_{Z^{\rm m}})$ and $(Z', \Delta_{Z'})$ are KLT. Since we have
    \[ K_Z + \Delta_Z + E \leq \pi^*(K_{Z'} + \Delta_{Z'} + E')  \]
    by the negativity lemma, if $E'$ is contracted on $Z'$, then we have $A_{(Z', \Delta_{Z'})} (E) \leq 0$, which contradicts the fact that $(Z', \Delta_{Z'})$ is KLT. Therefore, we see that $E'$ is not contracted on $Z'$.
    
    Let $\Delta_{Z'},E'$ denote the strict transforms of $\Delta_Z, E$ on $Z'$ respectively. Recall that since the relative minimal model $(Z, \Delta_Z + E)$ is dlt and since $E$ is prime, we know that $(Z, \Delta_Z + E)$ is PLT. Using
    \[ K_Z + \Delta_Z + E \leq \pi^*(K_{Z'} + \Delta_{Z'} + E')  \]
    again, we see that $(Z', \Delta_{Z'} + E')$ is PLT since $E'$ is not contracted on $Z'$. Now, since $- K_{Z'} - \Delta_{Z'} - E'$ is ample, we may pick a general effective $\Q$-divisor $D^* _{Z'} \sim_\Q -(K_{Z'}+ \Delta_{Z'} + E')$ such that $\text{Supp}(D_{Z'} ^*) $ does not intersect the non-snc locus of $(Z', \Delta_{Z'} +E')$ and such that $(Z', E'+ \Delta_{Z'} + D^* _{Z'}) $ remains PLT.  

    Finally, by taking strict transforms of $D^* _{Z'}$ on $Z$ and $S$ and using the fact that $(Z', E'+ \Delta_{Z'} + D^*_{Z'} ) $ is a PLT pair with $ K_{Z'} + E'+ \Delta_{Z'} + D^* _{Z'} \sim _\Q 0 $, we see that $(S, \Delta + D^*_S)$ is log canonical with a unique LC-place $E$. And since $K_Z + E + \Delta_Z + D^* _Z \sim_ \Q 0$, we also have $K_S + \Delta + D^*_S \sim _\Q 0$ by the negativity lemma. This completes the proof of parts (2) and (3), and hence of the theorem.
\end{proof}

\begin{rem}
    Since any one-dimensional KLT pair $(E, \Delta_E)$ must have $E$ normal, we see that in the proof above, $E \subset Z$ is smooth since $(Z, \Delta_Z + E)$ is PLT. Similarly, since a two-dimensional KLT pair $(S, \Delta_S)$ with standard coefficients is strongly $F$-regular when $p >5$ (see \cite{HaraDimensionTwo}), such pairs satisfy that $S$ is normal. Therefore, the surface $Z$ constructed in \Cref{thm:specialdivisor} is normal, and if either $\Delta$ has standard coefficients, or if $S$ is $\Q$-Gorenstein, then $S$ is normal as well. 
\end{rem}

\subsection{Global $F$-regularity of non-weakly exceptional log del Pezzo surfaces.}
Recall that $k$ was assumed to be an algebraically closed field of characteristic $p>5$.

\begin{thm} \label{thm:gFrofsurfaces}
    Let $(S, \Delta_S)$ be a log del Pezzo surface over $k$ where $\Delta_S$ has standard coefficients (i.e., contained in the set $\{0\} \cup \{\frac{n-1}{n} \, | \, n \geq 2 \}$). Suppose $(S, \Delta_S)$ is not weakly exceptional (\Cref{dfn:weaklyexcepFano}). Then, $(S, \Delta_S)$ is globally $F$-regular.
\end{thm}

\begin{proof}
    By \Cref{thm:specialdivisor}, we have a birational map $\mu: Z \to S$ such that $E \subset Z$ is a prime divisor and the pair $(Z, E + \Delta_Z )$ is PLT with $ - K_Z - \Delta_Z - E$ big (where $\Delta_Z$ is the strict transform of $\Delta$). Then, as in the proof of \Cref{thm:specialdivisor}, let $\pi: Z \to Z'$ be the $- K_Z - \Delta_Z - D_Z$-ample model. Then, $E'$ is not contracted on $Z'$ and $- K_{Z'} - \Delta_{Z'} - E'$ is ample, and we have $\Delta_{Z'}$ has standard coefficients.
    
    Now, by adjunction, we may write $K_{\overline{E'}} + \Delta_{\overline{E'}} = (K_{Z'} + E' + \Delta_{Z'})|_{\overline{E'}} $ where $\overline {E'} \to E'$ is the normalization map. Here, $\Delta_{\overline{E'}}$ is the different of $\Delta_{Z'}$ and so also has standard coefficients (\cite[Lemma 4.1]{HaconMckernanXuACCforLCT}). Now, $(E', \Delta_{E'})$ is a one-dimensional log Fano pair over $k$ with standard coefficients. Then, \cite[Theorem 4.2]{WatanabeFregularFpureRings} proves that $(E', \Delta_{E'})$ is globally $F$-regular. By global $F$-adjunction, we also get that $(Z', E' + \Delta_{Z'})$ is purely globally $F$-regular. Since $K_Z + \Delta_Z + E$ is $\pi$-nef, by the negativity lemma, we see that $K_Z + \Delta_Z + E \leq \pi^*(K_{Z'} + \Delta_{Z'} + E')$ (\Cref{lem:birationalcontraction}). Therefore, we see that $(Z, \Delta_{Z} + E)$ is purely-globally $F$-regular (\Cref{lem:GFRunderbirationalmap}), and hence $(S, \Delta_S)$ is globally $F$-regular.
\end{proof}

\begin{rem}
    Global $F$-regularity of log del Pezzo surfaces (with $\Delta = 0$, say) is known in the following (non-exhaustive list) cases:
    \begin{enumerate}
        \item If $p \geq 7$ and $X$ has canonical singularities by \cite{KawakamiTanakaGlobalFregularity}.

        \item For any $\varepsilon \in (0,1)$, if $X$ is $\varepsilon$-KLT and $p \geq p_0 (\varepsilon)$ is large depending on $\varepsilon$ by \cite[Section 3]{CTW17}.

        \item For any $p \geq 7$ and $\varepsilon < \varepsilon_0 (p)$ and $X$ is not $\varepsilon$-KLT by \cite[Section 5]{CTW17}.

        \item For any $p \geq 7$ and $X$ is not weakly-exceptional by \Cref{thm:gFrofsurfaces}.
    \end{enumerate}

    On the other hand, there exist non globally $F$-regular log del Pezzo surfaces in every characteristic $p>0$ by \cite{CasciniTanakaWitaszekcounterexamplestoFregularity}. The above list suggests that such surfaces are rare and it would be interesting to characterize this phenomenon.
\end{rem}

\subsection{PLT blow-ups in positive and mixed characteristics}

Now we use the results of the previous sections to study three-dimensional singularities in the local case. The main tool is the following:

\begin{defi} \label{defi:pltblowupinmixedchar}
       Let $x \in (X, \Delta)$ be a closed point, where $(X, \Delta)$ is a KLT pair over $V$. Then, a proper birational map $\pi: Y \to X$ over $V$ is called a \emph{PLT blow-up} of $(X, \Delta)$ (over $x$) if the following conditions hold:
    \begin{enumerate}
        \item $S \subset Y$ is the unique, prime $\pi$-exceptional divisor on $Y$ and $\pi(S) = \{x\}$. 
        \item $\pi$ is an isomorphism over $X \setminus \{x\}$.
        \item The pair $(Y, S + \Delta_Y)$ is PLT, where $\Delta_Y$ denotes the strict transform of $\Delta$ on $Y$.
        \item $-S$ is $\Q$-Cartier and ample over $X$.
    \end{enumerate}

    Given a PLT blow-up $\pi : Y \to X$ of $(X, \Delta)$ with exceptional divisor $S$, the pair $(S^{\text{N}}, \Delta_{S^{\text{N}}})$ is called a \emph{Koll\'ar component} of $(X, \Delta)$, where $S^{\text{N}}$ denotes the normalization of $S$ and $\Delta_{S^{\text{N}}} = \text{Diff}_{S^{\text{N}}} (\Delta_Y)$, the different of $\Delta_Y$ along $S$ as in \cite[Definition~4.2]{KollarKovacsSingularitiesBook}.
\end{defi}

With notation as above, we may write $K_Y + \Delta_Y + S = \pi^* (K_X + \Delta) + A_{X, \Delta} (S) S$. Using the fact that $-S$ is ample over $X$, we see that $K_Y + \Delta_ Y + S$ is anti-ample over $X$ since $A_{X, \Delta}(S) >0$.

\begin{rem}
    Given a PLT blow-up $\pi: (Y, S) \to X$ of a KLT pair $(X, \Delta)$ over $V$, note that since $\pi (S) $ is a closed point, $S$ is a variety over the residue field of $V$. Let $\nu: S^{\text{N}} \to S$ denote the normalization of $S$. Then, since the different $\Delta_{S^\text{N}} $ is defined by the formula
    \[ \nu^* (K_Y + \Delta_Y + S)|_S = K_{S^\text{N}} + \Delta_{S^\text{N}} \]
    and we see that $-(K_{S^\text{N}} + \Delta_{S^\text{N}} )$ is ample on $S^{\text{N}}$. Further, if the dimension of $X$ is at most $3$, PLT adjunction holds in our context (see \cite[Corollary 4.9]{HaconWitaszekRelativeMMPfor4folds}), and we conclude that $(S^\text{N}, \Delta_{S^\text{N}}) $ is a KLT log Fano pair over $k$ of dimension equal to $\dim(X) -1$. 
\end{rem}

\begin{rem} \label{rem:normalityofpltcenter}
    Moreover, if $X$ has dimension at most $3$ and the coefficients of $\Delta$ belong to the standard set $\{\frac{n-1}{n} \, | \, n \geq 1 \},$ then $\Delta_{S^\text{N}}$ also has coefficients in the standard set. In that case, $(S^\text{N}, \Delta_{S^\text{N}})$ is (locally) strongly $F$-regular by \cite[Theorem 3.1]{HaconXuThreeDimensionalMinimalModel}, and consequently by \cite[Proposition 4.1]{HaconXuThreeDimensionalMinimalModel}, $S$ is normal. Therefore, when $\Delta$ has standard coefficients and dimension of $X$ is at most $3$, we need not normalize and we will drop the superscript N from the Koll\'ar component.
\end{rem}

\begin{defi} \label{defi:weaklyexceptionalsing}
    Recall that a log Fano pair $(S, \Delta_S)$ over $V$ is said to be weakly exceptional if $(S, \Delta_S + \Gamma)$ is log canonical for every effective $\Q$-divisor $\Gamma$ such that $\Gamma \sim_\Q -(K_S + \Delta_S)$.
    
    A local KLT pair $x \in (X, \Delta)$ over $V$ is said to be weakly exceptional if \emph{every} Koll\'ar component $(S, \Delta_S)$ of $x \in (X, \Delta)$ is weakly exceptional.
    
    A non-weakly exceptional singularity over $V$ is a local KLT pair $x \in (X, \Delta)$ over $V$ that is not weakly exceptional.
\end{defi}

\begin{rem}
    Over $\C$, a weakly-exceptional singularity can also be characterized by the following property: given any two PLT blow-ups $\pi_1: Y_1 \to X$ and $\pi_2 : Y_2 \to X$ of $(X, \Delta)$, the birational map $\pi_1 ^{-1} \circ \pi_2 $ extends to an isomorphism in codimension one. In short, weakly exceptional singularities admit a unique Koll\'ar component, whereas non-weakly exceptional singularities admit more than one Koll\'ar component up to isomorphism. See \cite{ProkhorovBlowupsofcanonicalsingularity} for the details.
\end{rem}

We note that using the relative minimal model program, we can show that every KLT singularity of dimension at most $3$ admits at least one Koll\'ar component:

\begin{lem} \label{lem:existenceofKollarcomp}
    Let $x \in (X, \Delta)$ be an affine KLT singularity over $V$ of dimension at most $3$. Suppose $\Gamma \geq 0$ is an effective $\Q$-Cartier divisor such that $(X, \Delta + \Gamma)$ is log canonical, and KLT on $X \setminus \{x\}$. Assume that $(X, \Delta + \Gamma)$ is not KLT at $x$. Then, $x \in (X, \Delta)$ admits a PLT blow-up $\pi: Y \to X$ over $x$ with an exceptional divisor $S$ satisfying $A_{X, \Delta + \Gamma} (S) = 0$.
\end{lem}

\begin{proof}
   In positive characteristics, this follows from \cite[Proposition 2.15]{GongyoNakamuraTanakaRationalPoints} \cite[Theorem 9.37]{BMPSTWWMixedCharMMP}. The same proof also works in the mixed characteristic case, by using the required MMP with scaling as proved in \cite[Theorem~9.37]{BMPSTWWMixedCharMMP}. Note that for the tie-breaking argument in \cite[Proposition 2.15]{GongyoNakamuraTanakaRationalPoints}, we only need to apply the Bertini theorem for an ample divisor and regular projective schemes as in \cite[Theorem 2.17]{BMPSTWWMixedCharMMP}. See \cite[Proof of Theorem 9.27]{BMPSTWWMixedCharMMP} for a similar argument.
\end{proof}

\begin{cor} \label{cor:weaklyexceptionalKollarcomponent}
    If a local KLT singularity $x \in (X, \Delta)$ of dimension at most $3$ is non-weakly exceptional then there exists a Koll\'ar component $(S, \Delta_S)$ of $x \in (X, \Delta)$ such that $(S, \Delta_S)$ is not weakly exceptional.
\end{cor}

\subsection{$+$-regularity of $3$-dimensional non-weakly exceptional singularities when $p >5$.}

Now we generalize \Cref{thm:gFrofsurfaces} to three-dimensional singularities.

\begin{defi} \label{defiplusreg}
    Let $X$ be a normal scheme of finite type over $V$ and let $\Delta$ be an effective $\Q$-divisor over $X$. Then, the pair $(X, \Delta)$ is called globally $+$-regular if for every finite dominant map $f: Y \to X$ with $Y$ normal, the natural map
    \[ \cO_X \to  f_* (\cO_Y ( \lfloor f^* \Delta \rfloor))  \]
    splits as a map of $\cO_X$-modules. If $X$ is affine, we omit the adjective global.
\end{defi}

\begin{thm} \label{+regforthreedimsingularities}
    Let $x \in (X, \Delta)$ be a three-dimensional KLT singularity over $V$ with residue characteristic $p>5$. Suppose that $K_X$ is $\Q$-Cartier, $\Delta$ has standard coefficients and $(X, \Delta)$ is non-weakly exceptional (\Cref{defi:weaklyexceptionalsing}). Then, $(X, \Delta)$ is $+$-regular.
\end{thm}

\begin{proof}
    Since $(X, \Delta)$ is non-weakly exceptional, by \Cref{cor:weaklyexceptionalKollarcomponent} there exists a PLT blow-up $\pi:Y \to X$ (\Cref{defi:pltblowupinmixedchar}) such that the Koll\'ar component $(S, \Delta_S)$ is not weakly-exceptional. By \Cref{rem:normalityofpltcenter}, recall that in this case, $S$ is normal and $(S, \Delta_S)$ is a log Fano surface pair over $k$. Then by \Cref{thm:gFrofsurfaces}, we know that $(S, \Delta_S)$ is a globally $F$-regular log Fano pair.

    To show that $(X, \Delta)$ is $+$-regular at $x$, we may localize at $x$ and assume that $X = \Spec(R)$, where $(R, \fm)$ is a local ring. Moreover, by \cite[Corollary 6.9]{BMPSTWWMixedCharMMP} we may further base change $(R, \Delta)$ to the $\fm$-adic completion $(\widehat{R}, \Delta_{\widehat{R}})$. Note that if $\pi: Y \to X $ was the blow-up of an $\fm$-primary ideal $\fa \subset \fm$, then the blow-up of $\fa \widehat{R}$ on $\Spec(\widehat{R})$ defines a PLT blow-up $\widehat{\pi}: \widehat{Y} \to \widehat{X}$ with Koll\'ar component isomorphic to $(S, \Delta_S)$ (because $S \isom \Proj_R( \oplus_{m \geq 0} \, \, \fa^m/ \fa^{m+1})$).
    Since $(S, \Delta_S)$ is globally $F$-regular over $k$, by \cite[Theorem F]{MaSchwedeTuckerWaldronWitaszekAdjoint} we obtain that $(\widehat{R}, \Delta_{\widehat{R}})$ is BCM-regular, in particular $+$-regular.  This completes the proof of the theorem.
\end{proof}

\section{Inductive step} \label{section:inductivestep}

In this section we will establish a partial inductive approach to proving \Cref{conj:alpha-lim-positive} and \Cref{strongerconjversion}. Recall that a log Fano pair $(X, \Delta)$ is weakly exceptional if for every effective $\Q$-divisor $\Gamma$ on $X$ such that $\Gamma \sim _\Q -(K_X + \Delta)$, the pair $(X, \Delta+ \Gamma)$ is log canonical. In other words, we have $\alpha(X,\Delta) \geq 1$.

\begin{thm} \label{inductivethm}
Let $n \geq 2$ be an integer. Suppose we assume one of the following:
\begin{enumerate}
    \item \Cref{conj:alpha-lim-positive} holds for all log Fano pairs $(X, \Delta)$ of dimension $n-1$ over $\C$.
    \item \Cref{strongerconjversion} holds for all log Fano pairs $(X, \Delta)$ of dimension $n-1$ over $\C$.
    \item \Cref{conj:alpha-lim-positive} holds for all log Fano pairs $(X, \Delta)$ of dimension $n-1$ over $\C$ with the coefficients of $\Delta$ contained in the standard set $\{ \frac{m-1}{m} \, | \, m \in \N \}$.
    \item \Cref{strongerconjversion} holds for all log Fano pairs $(X, \Delta)$ of dimension $n-1$ over $\C$ with the coefficients of $\Delta$ contained in the standard set $\{ \frac{m-1}{m} \, | \, m \in \N \}$.
\end{enumerate}
Then, correspondingly, we have that
\begin{enumerate}
    \item \Cref{conj:alpha-lim-positive} holds for all non-weakly exceptional log Fano pairs $(X, \Delta)$ of dimension $n$ over $\C$.
    \item \Cref{strongerconjversion} holds for all non-weakly exceptional log Fano pairs $(X, \Delta)$ of dimension $n$ over $\C$.
    \item \Cref{conj:alpha-lim-positive} holds for all non-weakly exceptional log Fano pairs $(X, \Delta)$ of dimension $n$ over $\C$ with the coefficients of $\Delta$ contained in the standard set $\{ \frac{m-1}{m} \, | \, m \in \N \}$.
    \item \Cref{strongerconjversion} holds for all non-weakly exceptional log Fano pairs $(X, \Delta)$ of dimension $n$ over $\C$ with the coefficients of $\Delta$ contained in the standard set $\{ \frac{m-1}{m} \, | \, m \in \N \}$.
\end{enumerate}
\end{thm}

The proofs of all four parts of \Cref{inductivethm} follow the same strategy, so we will not prove them separately. The specific details that differ for the different parts will be highlighted in the course of the proof. The proof requires several steps. So we first give an informal summary of the key ideas involved. \\

\paragraph{\emph{Idea of the proof:}}
For simplicity, we assume that $\Delta =0$. Firstly, by \Cref{alphavss}, it is sufficient to prove that given a non-weakly exceptional Fano variety $X$ of dimension $n$, there exists a constant $\alpha > 0$ such that for all primes $p \gg 0$ and a characteristic $p$ model $X_p$, the following statement holds: for every effective $\Q$-divisor $D \sim_\Q -K_X $ on $X_p$, we have that $(X_p, \alpha D)$ is globally $F$-regular. The proof relies on finding a suitable divisor $E$ and a log Fano pair $(E, \Delta_E)$ of dimension $n-1$ from which we ``lift Frobenius splittings" via inversion of $F$-adjunction. The idea of finding such divisors comes from the results by Zhuang and Blum-Liu-Xu \cite{BLX22} and such divisors correspond to special test configurations of $X$. Roughly, such a divisor $E$ is a log canonical place of a $\Q$-complement $\Gamma$ of $(X, \Delta)$ similar to \Cref{thm:specialdivisor}, which exists if and only if $X$ is not weakly-exceptional. Suppose that $\mu: Y \to X$ denotes a birational map that uniquely extracts $E$ as a divisor. Then to apply global $F$-adjunction along $E$, we need $-K_Y - E$ to be ample. So we run an MMP to obtain a birational model $Y \dashrightarrow Y'$ such that after taking strict transforms, we get a $-K_{Y}- E'$-ample model. Now, we can lift Frobenius splittings from $E'$ to $Y'$ by global $F$-adjunction. Lastly, to lift the Frobenius splittings from $Y' $ back to $Y$, we use an idea from \cite{LiuXuZhuangHigherrankfinitegeneration} to perturb the divisor $D$ to make $D_Y$ proportional to $-K_Y - E$. In this case, we may lift the Frobenius splittings from $Y'$ to $Y$ (and then push-forward to $X$) as in \Cref{lem:GFRunderbirationalmap}. \qed

\bigskip

\paragraph{\textit{Proof of \Cref{inductivethm}}.} We divide the proof into five steps.
\medskip

\paragraph{\textbf{Step 1: Set up over $\C$.}} Let $(X, \Delta)$ be a log Fano pair over $\bC$ such that $\alpha(X, \Delta)<1$. By \cite[Lemma~6.1]{LiuMoragaSussboundedness23}, there exists a non-trivial special test configuration $(\cX, \cD)$ of $(X, \Delta)$ with central fiber $(X_0, \Delta_0)$ such that 
\begin{equation} \label{eqn:alpha0bound}
\alpha(X_0, \Delta_0)\geq \frac{1}{n}\min\{\alpha(X, \Delta), 1-\alpha(X, \Delta)\}.
\end{equation}
We refer to \cite[Definition~2.27]{XuKstabilitybook} for the definition of special test configurations.
Set $\alpha_0:=\alpha(X_0, \Delta_0)$. Recall that given a proper birational map $\mu: Y \to X$ and a prime divisor $E \subset Y$, we define \[ T_{X, \Delta} (E) = \sup \{ \lambda \geq 0 \, | \, \text{the divisor } \mu^* (-K_X - \Delta) - \lambda E \text{ is big} \} \]

The special test configuration $(\cX, \cD)$ is induced by some prime divisor  $E$ over $X$, in the sense of \cite[Definition~4.21]{XuKstabilitybook}. In particular, by \cite[Theorem 4.27]{XuKstabilitybook}, we know that $E$ satisfies the following properties: 
$A_{(X, \Delta)} (E)< T_{X, \Delta}(E)$,
and there exists a $\bQ$-complement $D^*$ of $(X, \Delta)$ such that $E$ is the unique LC-place of $(X,\Delta+D^*)$. Let $\mu:Y\to X$ be the $\Q$-factorial dlt modification of $(X,\Delta + D^*)$ which only extracts $E$ (which exists by \cite[Corollary 1.4.3]{BCHM}). In other words, $\mu$ is a birational morphism from a $\Q$-factorial projective variety $Y$, $E \subset Y$ is a prime divisor and the exceptional locus of $\mu$ is equal to $E$ (if $\mu$ is not an isomorphism). Let $\Delta_Y$ and $D_Y ^*$ denote the strict transforms of $\Delta$ and $D^*$ respectively. 
Then, since $E$ is an LC-place of $(X, \Delta+D^*)$ and $D^* \sim _\Q  -K_X - \Delta$, we know that
\begin{equation} \label{YDElcy} K_Y + \Delta_Y+ D^* _Y + E = \mu ^* (K_X + \Delta +  D^*) \sim _{\Q} 0.   \end{equation}
Moreover, since $E$ is the unique LC-place of $(X, \Delta+ D^*)$, we see that $(Y, E+ \Delta_Y + D_Y^*)$ is a PLT log Calabi-Yau pair. Moreover, because $A_{(X, \Delta)}(E) < T_{X, \Delta} (E)$ by construction, we have that $-K_Y - \Delta_Y -E =\mu^*(-K_X - \Delta) - A_{(X, \Delta)}(E) E$ is big.  

Next, we run a $(-K_Y-\Delta_Y-E)$-MMP which terminates to $Y\dashrightarrow Y^{\rm m}$ such that, if $E^{\rm m} $, $\Delta_{Y^{\rm m}}$ denotes the strict transform of $E$ and $\Delta_Y$ on $Y^{\rm m}$, we have $-K_{Y^{\rm m}} -\Delta_{Y^{\rm m}} -E^{\rm m}$ is big and nef. We can run such an MMP because $Y$ is a $\Q$-factorial variety of Fano type by \cite[Lemma 3.2]{BLX22}. Let us denote the steps of this MMP by
\[ f_i : Y_i \dashrightarrow Y_{i+1} \quad \text{for $0 \leq i \leq N$},\]
 where $Y_0 = Y$ and $Y_N = Y^{\rm m}$. We may assume that each $f_i$ is either an isomorphism in codimension $1$ (i.e.\ a flip), or  a divisorial contraction, in which case $f_i$ is a birational morphism. Furthermore, since $Y$ is of Fano type, by \cite[Lemma~2.12]{BirkarAntipluricanonicalsystems} and \Cref{codimensiontwolemma}, we see that each $Y_i$ is of Fano type. 
 Moreover, let $f_{\rm m}: Y^{\rm m} \to Y'$ be the ample model of $-K_{Y^{\rm m}} - \Delta_{Y^{\rm m}} - E^{\rm m}$. By \cite[Lemma~2.12]{BirkarAntipluricanonicalsystems} again, $Y'$ is also of Fano type.

Next, inductively defining $E_{i+1} = (f_i)_* E_i$, $\Delta_{i+1} = (f_i)_* \Delta_i$ and $D_{i+1} ^* = D_{i} ^* $, we claim that each $E_i$ is non-zero. First we show that each $(Y_i, \Delta_i)$ and $(Y' , \Delta_{Y'})$ is a KLT pair. For this, pick a rational number $ 0 <\varepsilon \ll 1$ and a general divisor $H \in | - (K_X + \Delta + (1 - \varepsilon) D^*) \sim -\varepsilon (K_X + \Delta)|_\Q$. Then, it follows from \Cref{YDElcy} that for all such $0 < \varepsilon \ll 1$ and $H$ general, we have $(X, \Delta + (1 - \varepsilon) D^* + H)$ is KLT and $K_X + \Delta + (1 - \varepsilon) D^* + H \sim _\Q 0 $. Therefore, by crepant pull-back to $Y$, $Y_i$'s and $Y'$, we get KLT Calabi-Yau pairs dominating $(Y_i, \Delta_i)$ and $(Y', \Delta_{Y'})$. Therefore, each of $(Y_i, \Delta_i)$'s and $(Y', \Delta')$ must be KLT.

Next, observe that since $(Y, \Delta_Y + E + D^{*}_Y)$ is a PLT log Calabi-Yau pair, the pairs $(Y_i, \Delta_{Y_i} + E_i + D^{*}_{Y_i})$ are at least log canonical. Inductively, we assume that $E_i \neq 0$ and it is enough to consider the case when $f_i$ is a divisorial contraction.
Now since any $-K_Y - \Delta_Y  - E$-MMP is the same as $\varepsilon D^* _ Y \sim K_ Y + \Delta_Y + E + (1+ \varepsilon) D^* _Y$-MMP for any $0 < \varepsilon  \ll 1$, by the negativity lemma we have
\begin{equation} \label{pltlcyequation} K_{Y_i} + \Delta_i + E_i + (1 + \varepsilon) D^* _{Y_i} - f_i ^* (K_{Y_{i+1}} + \Delta_{i+1} + E_{i+1} + (1+ \varepsilon) D_{Y_{i+1}} ^*) \geq  0. \end{equation}
Therefore, if $E_{i+1} = 0$, we have $A_{(Y_{i+1}, \Delta_{i+1} + ( 1 + \varepsilon)D^* _{Y_{i+1}})} (E_i) \geq 0$. On the other hand, since $K_{Y_i} + \Delta_{Y_{i}} + E_i$ is $f_i$-nef, we see that
\[  K_{Y_i} + \Delta_i + E_i - f_i ^* (K_{Y_{i+1}} + \Delta_{i+1} + E_{i+1}) \leq  0.  \]
  So, if $E_{i+1} = 0$ then we see that $A_{(Y_{i+1}, \Delta_{i+1})}(E_i) \leq 0$. This forces all inequalities in
  \[0 \leq A_{(Y_{i+1}, \Delta_{i+1} + ( 1 + \varepsilon)D^* _{Y_{i+1}})} (E_i) \leq A_{(Y_{i+1}, \Delta_{i+1})}(E_i) \leq 0 \]
  to be equalities. In particular, this is a contradiction to the fact that $(Y_{i+1}, \Delta_{i+1})$ is KLT. Thus, we see that each $E_i \neq 0$ on $Y_i$.

  Similarly, we check that $E^{\rm m}$ is not contracted in the ample model $Y'$. Since $Y'$ is the ample model of $-K_{Y^{\rm m}} - \Delta_{Y^{\rm m}} - E^{\rm m}$, we must have $ K_{Y^{\rm m}} +\Delta_{Y^{\rm m}} + E^{\rm m}  = f_{\rm m} ^*(K_Y' + \Delta_{Y'} + E')$ because $K_{Y^{\rm m}} +\Delta_{Y^{\rm m}}+  E^{\rm m}$ is relatively numerically trivial over $Y'$. Thus, if $E' = 0$, we see that $(Y', \Delta_{Y'})$ is not KLT because $E^{\rm m}$ will be an LC-place of $(Y', \Delta_{Y'})$. But this is a contradiction since we have shown that $(Y', \Delta_{Y'})$ must be KLT.

Lastly, let $\pi: Y \dashrightarrow Y'$ denote the composite map. Since $(Y,E+ \Delta_Y + D_Y^*)$ is a PLT log Calabi-Yau pair, so is $(Y', E' + \Delta_{Y'} + D_{Y'}^*)$ since $E'$ is not contracted on $Y'$. In particular, we know that $(Y', \Delta_{Y'} + E')$ is a PLT log Fano pair. Denote by $\Delta_{E'}$ the different divisor of $(Y',\Delta_{Y'}+E')$ on $E'$, then $(E',\Delta_{E'})$ is a KLT log Fano pair. Note that if $\Delta$ has standard coefficients, then $\Delta_{E'}$ also has standard coefficients (i.e., coefficients in the standard set $\{ \frac{n-1}{n} \,| \, n \in \N _{>0} \}$, see \cite[Lemma~4.1]{HaconMckernanXuACCforLCT}).

Before moving on to the next step, we record the following lemma that will be useful in obtaining effective versions of \Cref{inductivethm}.

\begin{lem} \label{lem:alphaofE'}
    With notation as in Step 1 above, we have $\alpha(E', \Delta_{E'}) \geq \alpha (X_0 , \Delta_0) = \alpha_0.$
\end{lem}

\begin{proof}
    Suppose $\Gamma \geq 0$ is an effective divisor on $E'$ such that $\Gamma \sim_\Q - (K_{E'} + \Delta_{E'})$. Then, we claim that there is an effective divisor $\Gamma_{Y'} \sim_{\Q} - K_{Y'} - \Delta_{Y'} - E'$ such that $\Gamma_{Y'}$ does not have $E'$ as a component and $\Gamma_{Y'} |_{E'} = \Gamma $. This follows from the observation that for any integer $r \gg 0$ divisible enough, the restriction map $H^0 (Y', \cO_{Y'} (-r(K_{Y'} + \Delta_{Y'} + E'))) \to H^0 (E', \cO_{E'} (-r(K_{E'} + \Delta_{E'})))  $ is surjective. Indeed, the next term in the long exact sequence for cohomology is $H^1 (Y', \cO_{Y'} (L))$, where $L = -r(K_{Y'} + \Delta_{Y'} + E') - E' \sim_\Q K_{Y'} + \Delta_{Y'} - (r+1) (K_{Y'} + \Delta_{Y'} + E'). $ This cohomology group vanishes by Kawamata-Viehweg vanishing since $(Y', \Delta_{Y'})$ is KLT and $ - K_{Y'} - E' - \Delta'$ is ample.

    Now, since $K_{Y'} + \Delta_{Y'} + E' + \Gamma_{Y'} \sim_ \Q 0$, there is an induced divisor $\Gamma_Y \sim_\Q  -K_Y - \Delta  - E$ such that the strict transform of $\Gamma_Y$ is $\Gamma_{Y'}$. By pushing-forward along the map $\mu$, there is a divisor $\Gamma_X \sim _\Q - K_X - \Delta $ such that $\ord_E (\Gamma_X) = A_{(X, \Delta)} (E)$ and $\Gamma _Y $ is the strict transform of $\Gamma_X$. Now, applying \cite[Lemma~4.9]{LiuXuZhuangHigherrankfinitegeneration}, there is a divisor $\Gamma_X ^+ \sim_\Q -(K_X + \Delta)$ such that $\Gamma_X ^+ \geq \alpha_0 \Gamma_X$, $(X, \Delta + \Gamma_X ^+)$ is log canonical with $E$ as an LC-place. In other words, by taking strict transforms, we have $K_Y + \Delta_Y + E + \Gamma_{Y} ^+ \sim_\Q (K_X +\Delta + \Gamma_X ^+) \sim_\Q 0$. Therefore, $(Y, \Delta_Y + E + \Gamma_Y ^+)$ is a log-canonical log Calabi-Yau pair. Taking strict transforms again, we obtain that $(Y', \Delta_{Y'} + E' + \Gamma_{Y'} ^+)$ is a log canonical log Calabi-Yau pair. Moreover, since $\Gamma _X ^+ \geq \alpha_0 \Gamma_X$, we have $\Gamma_{Y'} ^+ \geq \alpha_0 \Gamma_{Y'}$. Note that $\Gamma_{Y'} ^+$ does not have $E'$ as a component, since $\Gamma_Y ^+$ was the strict transform of $\Gamma_X ^+$. Therefore, by adjunction, we see that $(E', \Delta_{E'} + \Gamma^+)$ is log canonical log Calabi-Yau pair and $\Gamma_{Y'} ^+ \geq \alpha_0 \Gamma$. Therefore, we see that $(E', \Delta_{E'} + \alpha_0 \Gamma)$ is log-canonical and therefore, $\alpha (E', \Delta_{E'}) \geq \alpha_0.$
\end{proof}

\medskip

\paragraph{\textbf{Step 2: Reduction to characteristic $p \gg 0$.}} Now, we reduce the varieties $X$, $Y_i$, $Y'$, the divisors $\Delta$, $\Delta_i$, $\Delta_{Y'}$, $D^*$, $D^* _ i$ and $D^* _{Y'}$, $E_i$, $E'$ and $\Delta_{E'}$, and the morphisms $ \mu, f_i, \pi $ to characteristic $p \gg 0$ over a finitely generated smooth $\Z$-algebra $A$, as defined in \Cref{reductindfn}. In this situation, after replacing $A$ by a finitely generated extension if necessary, we may assume that over the residue field $\kappa (s)$ of each closed point $s$ of $\Spec(A)$, the following conditions are satisfied:
\begin{enumerate}
    \item The reduction $X_s$ of $X$ is a geometrically integral projective variety over $\kappa(s)$, and the pair $(X_s, \Delta_s)$ is a globally $F$-regular log Fano pair (\cite[Theorem~5.1]{SchwedeSmithLogFanoVsGloballyFRegular}).
    \item The reduction $E_s$ of $E$ is a geometrically integral prime divisor on $Y_s$ and the reduction $\mu_s: Y_s \to X_s$ uniquely extracts $E_s$. We may also assume that the log discrepancy of $E_s$ is the same as $A_X(E)$, the log discrepancy over $\C$.
    \item The reductions of each $Y_i$, and $Y'$ are (geometrically integral) varieties of Fano type over $\kappa(s)$ (again by \cite[Theorem 5.1]{SchwedeSmithLogFanoVsGloballyFRegular}). Moreover, we may assume that each $f_i$ reduces to a step of the $-K_Y -\Delta_Y -E$-MMP over $\kappa(s)$. In particular, if $f_i$ is a flip, then $f_{i,s}$ is an isomorphism in codimension $1$, and if $f_i$ is a morphism with $K_{Y_i} + \Delta_i + E_i$ being $f_i$-nef, then we may assume that $K_{Y_{i,s}}+\Delta_{i,s}+E_{i,s}$ is $f_{i,s}$-nef. The last part can be done thanks to \Cref{lemmareductionofnef}. We may also assume that the reduction of $Y'$ to $\kappa (s) $ is the $-K_{Y_s} - \Delta_{Y_s} - E_s$-ample model of $Y^{\rm m} _s$.
    \item We may assume that the reduction of $(Y', \Delta_{Y'}+ E')$ is a globally purely $F$-regular pair over $\kappa(s)$, and the index of $K_{Y'_s} + \Delta_{Y'_s} + E'_s$ is not divisible by $p_s = \text{char}(\kappa(s))$. Furthermore, we may assume that the reduction of $\Delta_{E'}$ is the $F$-different for the pair $(Y' _s, \Delta_{Y' _s} + E'_s)$ (by the main theorem of \cite{das_different_different_different}).
    
    \item By \cite[Lemma~2.22]{GongyoOkawaSannaiTakagiCoxRings}, we may assume each of the varieties $X_s$, $Y_{i,s}$ and $Y'_s$ is a $\Q$-factorial Mori-dream space over $\kappa(s)$. In particular, we may assume that the Cox rings of $X_A$ and $Y_A$ are finitely generated and flat over $A$ and commute with base-change to fibers. In other words, we have
    \[ \rm{Cox} (X_A) \otimes _A \kappa(s) = \rm{Cox}(X_s),  \]
    and similarly for $Y$. In particular, we may assume that the section ring $S_A = S(X_A, -r(K_{X_A} + \Delta_A))$ is flat over $A$ and commutes with base-change to fibers.

    \item Let $S_\C$ denote the section ring $S(X, -r(K_X + \Delta))$ (over $\C$) where $r$ denotes the Cartier index of $-K_X - \Delta$, and let $S_A$ denote the corresponding model of $S_\C$ over $A$. Let $\mathscr{F} _E$ denote the filtration induced by $E$ on $S_\C$. More precisely, for any $\lambda,m \in \Z$, we have
    \[ \mathscr{F}_E ^\lambda S_m := \mu_*H^0 (Y, \cO_Y (-mr(K_X + \Delta_X) - \lambda E) ) \subset S_m. \]
    Then, we may assume that the filtration induced by $E_A$ on $S_A$, denoted by $\mathscr{F}_{E_A}$ on $S_A$ satisfies the following:
    \begin{enumerate}
        \item [(a)] $\mathscr{F}_{E_A}$ induces $\mathscr{F}_E$ on $S_\C$. In other words, for each $\lambda,m \geq 0$, we have $ \mathscr{F}_E ^\lambda S =  \mathscr{F}_{E_A} ^{\lambda} S_A \otimes_A \C. $
        \item [(b)] The filtration $\mathscr{F}_{E_s}$ on $S_{\kappa(s)}$ induced by $E_s$ is reduced from $\mathscr{F}_{E_A}$. In other words, we have
        \[ \mathscr{F}_{E_s} ^\lambda S_{\kappa(s)} =  \mathscr{F}_{E_A} ^{\lambda} S_A \otimes_A \kappa(s). \]
        \item [(c)] For each $\lambda, m \geq 0$, the associated graded $A$-modules defined by
\begin{equation}
   \text{gr} \, \mathscr{F}_{E_A} ^{\lambda,m} S = \mathscr{F}_{E_A} ^{\lambda} S_m/ \mathscr{F}_{E_A}  ^{> \lambda} S_m
\end{equation}
is flat (equivalently, projective) over $A$, and commute with base change to $\kappa (s)$ (i.e., $\text{gr} \, \mathscr{F}_{E_A} \otimes _A \kappa(s) = \text{gr} \, \mathscr{F}_{E_s}. $)
    \end{enumerate}
    \medskip
Parts (a) and (b) follow immediately from Part (5) of Step (2) above. For (c), recall that the filtration $\mathscr{F}_{E}$ is a finitely generated $\Z$-filtration on $S$. Thus, if we consider the bigraded Rees algebra
\[  \cR = \bigoplus_{m \in \Z_{\geq 0}} \bigoplus_{\lambda \in \Z} \mathscr{F}_{E} ^\lambda S_m \,  t^{-m} s^{-\lambda}, \]
then $\cR$ is finitely generated as a $\C$-algebra. Moreover, we also consider the associated graded algebra
\[  \text{gr} \, \cR := \cR / \cR[0, 1] =  \bigoplus_{m \in \Z_{\geq 0}} \bigoplus_{\lambda \in \Z} \text{gr}_{E} ^{\lambda,m} \,  t^{-m} s^{-\lambda}\]
which is also finitely generated over $\C$ since $\cR$ is. By generic flatness, \cite[\href{https://stacks.math.columbia.edu/tag/052A}{Tag 052A}]{stacks-project}, we may choose models for $\cR_A$ and $\text{gr} \, \cR _A$, so that we have
$$ \text{gr} \, \cR_s := \cR_s/ \cR_s [0,1] = ( \cR_A/ \cR_A[0,1]) \otimes_A \kappa(s), $$
where $\cR_s$ denotes the base change $\cR_A \otimes_A \kappa(s)$. 
    \end{enumerate}

\medskip

With this setup, the following is the main claim of the proof:
\medskip

\paragraph{\textbf{Main Claim:}} \textit{There exists a $p_0 \in \bN$ such that for any closed point $s \in \Spec(A)$ with residue characteristic larger than $p_0$, we have} 
\begin{equation} \label{equation:mainclaim}
\alpha_F (X_s, \Delta_s) \geq \alpha_0  \, \alpha_{F} (E_s, \Delta_{E_s}), \end{equation}   
where we recall that $\alpha_0$ was defined to be $\alpha(X_0, \Delta_0)$ in Step 1.
\bigskip

\paragraph{\textbf{Step 3: Construction of $D^+$.}} Fix a closed point $s \in \Spec(A)$ as above and let $p$ denote the characteristic of the residue field $\kappa(s)$. Note that $\kappa(s)$ is a finite field, hence is perfect.

Recall that $S_A$ is the section ring $S(X_A, -rK_{X_A})$. By localizing $A$ at the maximal ideal corresponding to $s$, we may assume that $A$ is a regular local ring. Let $S_{\kappa(s)}$ denote the base-change of $S_A$ to $\kappa(s)$.

\begin{lem} \label{lem:existenceofD+}
    Let $D$ be an effective $\Q$-divisor on $X_s$ such that $D \sim_{\Q} -K_{X_s} - \Delta_s$. Assume that $D$ is of the form $\text{div}(f_s ^t)$ for some section $f \in S_{\kappa(s)}$ and some rational number $t > 0$.
Then, there exists a $\Q$-divisor $D^+$ on $X_s$ satisfying the following properties: (a) $D^+ \geq \alpha_0 D$, (b) $D^{+} \sim_\Q -K_{X_s} - \Delta_s$, and, (c) $\ord_{E_s} (D^+) = A_{X} (E)$.
\end{lem}
\begin{proof}
Let $m$ be an integer such that $mrD$ is Cartier and let $f_s$ be as in the definition of $D$, so that $f_s$ is a global section in $H^0 (X_s, \cO_{X_s}(-mr(K_{X_s}+ \Delta_s))). $ So, to show that there exists such a $D^+$, it is equivalent to showing that for all $m'$ divisible enough (so that $m' \alpha_0$ is an integer), setting $N = m' m r A_X (E)$, there is a section $g \in S_{\kappa(s)}$ of degree $m m'$ such that $g$ is divisible by $f^{ \alpha_0 m' }$ in $S_{\kappa(s)}$ and $\ord_{E_s} (g) = N$.

Now, we fix a lift $f_A$ of $f_s$ to $S_{A}$ which is homogeneous of degree $m$. Note that we may lift any such $f_s$ to $A$ by Part (5) of Step 2. Further, for any $A$-algebra $B$, let $f_B$ denote the extension of $f_A$ to $B$.

    Using \Cref{fromCtocharp} proved below, to show the existence of such a $g \in S_{\kappa(s)}$, it is sufficient to show the same statement over $\C$. Over $\C$, this is guaranteed by \cite[Lemma 4.9]{LiuXuZhuangHigherrankfinitegeneration}.
\end{proof}

\begin{lem} \label{fromCtocharp}
    For any $m'$ divisible enough, suppose that there exists a section $g_\C \in S_{\C,m m'}$ such that $g_\C$ is divisible by $f_{\C}^{ \alpha_0 m' }$ and $\ord_{E_\C} (g_\C) = N$. Then, there exists $g \in S_{m m'}$ such that $g$ is divisible by $f^{ \alpha_0 m' }$ in $S$ and $\ord_{E_s} (g) = N$. 
\end{lem}

\begin{proof}
    Fix $m'$ divisible enough so that $m' \alpha_0$ is an integer. For any $A$-algebra $B$,  consider the following $B$-submodules of $ S_{B, mm'} := S_{A, mm'} \otimes_A B$:
    \[ M_B := f_B ^{m' \alpha_0} S_A \cap S_B = \text{im} (M_A \otimes_A B \to S_B),\]
    and similarly,
    $N_B = (\cF ^{N } _{E_A} S_{A, mm'}) \otimes _A B$, and $N'_B = (\cF _{E_A} ^{> N } S_{mm'}) \otimes_A B$. Recall that since we know that the associated graded modules of the filtration $\cF_{E_A}$ are flat over $A$, we have $N_B/N'_B = (N_A/N'_A) \otimes_A B$. Let $\psi_B$ denote the following natural map of $B$-modules:
    \[ \psi_B: M_B \cap N_B \to N_B/N'_ B. \]
    Note that with this notation, the lemma is equivalent to the statement that if the image of $\psi_ \C$ is non-zero then the image of $\psi _{\kappa(s)}$ is non-zero.

    Next, suppose $A \to k \to K$ is an extension such that $k$ and $K$ are both fields. Then, note that the image of $\psi_k$ is non-zero if and only if the image of $\psi_K$ is non-zero. In fact, we have $\psi_K = \psi_k \otimes _k K$. This is because taking intersections of subspaces and images of maps of vector spaces commute with base-change over the field. Now, setting $L = \text{Frac}(A)$, we apply the above observation to the extension $A \to L \to \C$ to conclude that if the image of $\psi_\C$ is non-zero, then, the same is true for $\psi_L$ as well.
    
    Next, we blow up the regular local ring $A$ at the maximal ideal $\fm$, and let $(A', t, K)$ be the local ring at the generic point of the exceptional divisor of the blow-up. Then, note that $A'$ is a discrete valuation ring with uniformizer $t$ such that $\fm A' = (t)$. Therefore, the residue field $K$ is an extension of $k = A/\fm$, and thus is of characteristic $p$ as well. Applying the above observation to the extension $k \to K$, we see that it is enough to show that the image of $\psi_K$ is non-zero.

    Now, since we know that the image of $\psi_L$ is non-zero, and $L =\text{Frac}(A')$, there exists elements $g \in S_{A', mm'}$, $h \in S_{A', mm' - \alpha_0 m'}$ and an integer $u$ such that in $S_L$, we have
    \[ g = f_{A'} ^{\alpha_0 m'} \, \frac{h}{t^{u}}, \]
    and $g$ is contained in $N_L \setminus N' _L$. Furthermore, by increasing $u$ if required, we may assume that the $t$-adic valuation of $g$ is zero. Thus, we have
    \[ g = f_{A'} ^{ \alpha _0 m'} \frac{h}{t^u}. \]
    Now, since by construction $f_k ^{\alpha_0 m'}$ is non-zero and the ring $S_K$ is an integral domain (since it is normal and $\N$-graded), we have 
    \[ 0 \neq f^{\alpha_0  m'} _{K} = f^{\alpha_0  m'} _{A'} \bmod{t}, \] 
    the $t$-adic valuation of $\frac{h}{t^{u}}$ is zero. Since $h$ is homogeneous, we may assume that $u = 0$. So we get that 
    $g = g = f_{A'} ^{ \alpha _0 m'} h$
    where $g \in S_{A', mm'}$ is non-zero modulo $t$. Therefore, the image of $g$ modulo $t$ is a non-zero element of $M_K$. Furthermore, since 
    $S_{A', mm'}/N_{A'}$ is free over $A'$, and $g$ is contained in $N_L \setminus N'_L$, $g$ must be contained in $N_{A'} \setminus N' _{A'}$. Finally, the image of $g$ modulo $t$ must be contained in $N_K/N'_K$, since we know that 
    \[ N_K/N'_K = \left(N_{A'}/N'_{A'} \right) \otimes _{A'} K =\left( \left(N_{A}/N'_{A} \right) \otimes _{A} k \right) \otimes_k K.  \]
    Since the image of $g$ modulo $t$ was already contained in $M_K$ by construction, we have that the image of $\psi_K$ is non-zero. This completes the proof of the lemma. 
\end{proof}

\paragraph{\textbf{Step 4: Proof of Main Claim \ref{equation:mainclaim}.}} Let $f \in S_{\kappa(s)}$ be a non-zero homogeneous element of degree $m$, and let $D$ denote the corresponding divisor $\text{div}(f^{\frac{1}{mr}})$ on $X_s$. Then note that $D \sim_\Q -(K_{X_s} + \Delta_s)$. Using \Cref{lem:existenceofD+}, there exists a $\Q$-divisor $D^+$ on $X_s$ such that $D^+ \geq \alpha_0 D$, $D^{+} \sim_\Q -K_{X_s} - \Delta_s$ and $\ord_{E_s} (D^+) = A_{X} (E)$.

For convenience, we drop the subscript $s$ below, since $s$ is fixed for the rest of the argument. Let $(Y, \Delta_Y+ E+D_Y^+)$ be the crepant pull-back of $(X, \Delta + D^+)$, i.e., if $D_Y ^+$ denotes the strict transform of $D^+$ on $Y$, we note that
\[ K_Y + \Delta_Y+ E + D^+ _Y = \mu^* (K_X + \Delta + D^+) \sim_\Q 0 \]
by our choice of $D^+$. Let $D^+ _{Y_i}$ denote the strict transform of $ D^+ _{Y}$ on $Y_i$, so that since $D^+ _Y \sim_\Q -(K_Y + \Delta_Y + E)$, we have $D^+ _{Y_i} \sim_\Q -(K_{Y_i} + \Delta_{Y_i} + E_i)$
by Part (1) of \Cref{lem:birationalcontraction}.
 Let $D_{E'}^+:=D_{Y'}^+|_{E'}$ and note that $D^+ _{Y'} \sim_\Q -(K_{Y'} + \Delta_{Y'}+ E') $ is $\Q$-Cartier and ample. We have $D_{E'} ^+ \sim _\Q -(K_{E'} + \Delta_{E'})$ by the definition of $\Delta_{E'}$.

%Now, we may base-change the above set up to the algebraic closure $\overline{\kappa(s)}$, and work over $k := \overline{\kappa(s)}$.

Now, to prove the Main Claim \ref{equation:mainclaim}, it suffices to show that for each rational number $0 < c < \FA(E', \Delta_{E'})$ and any $D$ as above, we have $(X, \Delta+ c\alpha_0 D)$ is globally $F$-regular. To see this, we first note that by the definition of the $\FA$-invariant, the pair 
$(E',\Delta_{E'} + cD^+ _{E'})$ is globally $F$-regular. By global $F$-adjunction (\Cref{lem:Fadjunction}), we know that $(Y', E' + \Delta_{Y'}+ cD^+_{Y'})$ is globally purely $F$-regular as well.

Now we will prove that $(Y_i,  \Delta_{Y_i} + cD_{Y_i} ^+)$ is globally $F$-regular for each $i \geq 0$ by a descending induction on $i$. Suppose we know that $(Y_{i+1},  \Delta_{Y_{i+1}} + cD_{Y_{i+1}} ^+)$ is globally $F$-regular. Then, the map $f_i: Y_i \dashrightarrow Y_{i+1}$ is either a birational contraction morphism or an isomorphism in codimension $1$. In the latter case, we may directly apply Part (1) of \Cref{lem:GFRunderbirationalmap}. But if $f_i$ is a birational contraction morphism, then note that since we have
\[K_{Y_i} + E_i+ \Delta_{Y_i}+ cD^+ _{Y_i} \sim_\Q (1-c) (K_{Y_i} +\Delta_{Y_i}+ E_i), \] 
we know that $K_{Y_i} + E_i+ \Delta_{Y_i}+ cD^+ _{Y_i}$ is $f_i$-nef. In this case, we may apply Part (2) of \Cref{lem:birationalcontraction} and Part (3) of \Cref{lem:GFRunderbirationalmap} to obtain that $(Y_i, \Delta_{Y_i} + cD_{Y_i}^+)$ is globally $F$-regular. This shows that $(Y, \Delta_Y + cD_Y ^+)  $ is globally $F$-regular. Finally by applying Part (2) of \Cref{lem:GFRunderbirationalmap}, this  implies that $(X, \Delta+cD^+)$ is globally $F$-regular and hence that $(X, \Delta + c\alpha_0  D)$ is globally $F$-regular as required. Since $D$ was an arbitrary $\bQ$-divisor $\Q$-linearly equivalent to $-K_{X_s} - \Delta_{s}$, this shows that
$\alpha_F(X_s, \Delta_s) \geq c \alpha_0$. This completes the proof of the main claim \Cref{equation:mainclaim}.

\medskip

\paragraph{\textbf{Step 5: Completion of the proof of \Cref{inductivethm}:}} 
To complete the proof of \Cref{inductivethm}, we note that if \Cref{conj:alpha-lim-positive} holds in dimension $n -1$, then there exists a constant $C >0 $ such that $\alpha_{-} (E', \Delta_{E'}) \geq C$. In other words, there exists a $p_1 >0 $ such that for all $p \geq p_1$, there is some $s \in \Spec(A)$ such that $ \FA (E'_s , \Delta_{E'_s}) \geq C $. Then, by the Main Claim \Cref{equation:mainclaim}, for the same points $s \in \Spec(A)$, we have $\FA(X_s, \Delta_s) \geq C \alpha_0$. By the definition of $\alpha_-$, we see that $\alpha_- (X, \Delta) > 0 $ as well, implying \Cref{conj:alpha-lim-positive} for $(X, \Delta)$.

Similarly, if \Cref{strongerconjversion} holds in dimension $n -1$, then there exists a constant $C >0 $ and an open set $U \subset \Spec(A)$ such that for all closed points $s \in U$, we have $\FA (E' _s, \Delta_{E'_s}) \geq C$. Then, by the Main Claim \Cref{equation:mainclaim} again, for all such $s \in U$, we have $\FA(X_s, \Delta_s) \geq C \alpha_0$, implying \Cref{strongerconjversion} for $(X, \Delta)$. Since we have already noted that $(E' , \Delta_{E'})$ has standard coefficients if $(X, \Delta)$ does, 
this completes the proof of \Cref{inductivethm}. \qed

The following results were used in the proof above.

\begin{lem} \label{lemmareductionofnef}
    Let  $f: Y \to X$ be a projective birational map between $\Q$-factorial, projective Mori dream spaces. Suppose $D$ is a divisor on $Y$ that is $f$-nef. Then, for any family of characteristic $p \gg 0$ models for $f, Y, X, D$ over $A$, after replacing $A$ by a finitely generated extension if necessary, we may assume that for all closed points $s \in \Spec(A)$ the reduction $D_s$ is $f_s$-nef.
\end{lem}

\begin{proof}
Let $-E = D - f^* (f_*D)$. Then it is sufficient to prove the lemma after replacing $D$ by $E$. Thus, we may assume that $D = -E$ is an $f$-exceptional divisor. By the negativity lemma, we know that $E$ is effective.

Now we pick an ample divisor $H$ on $X$. Since $-E$ is relatively nef, by replacing $H$ with a multiple if necessary, we may assume that $L := f^* H - E$ is nef on $Y$. Since $Y$ is a $\Q$-factorial Mori dream space, we know that since $L$ is nef, it is also semi-ample on $Y$. Let $g : Y \to Z$ be the morphism given by a multiple $mL$ with $ Z = \im (g)$. Since $E$ is effective, we know that $Z$ is a subvariety of $X$. Thus, we have
\[ g: Y \to Z \hookrightarrow X. \]
Thus, we see that $-E$ is $f$-semiample over $X$.

Lastly, given a family of characteristic $p$ models of $f$ over a finitely generated $\Z$-algebra $A$, after replacing $A$ bey a finitely generated extension, we may assume that we have a family of characteristic $p \gg 0$ models of $Z$, $g$ and $L$ as well. Moreover, by \cite[Lemma~2.22]{GongyoOkawaSannaiTakagiCoxRings}, we may assume that for any closed point $s \in \Spec(A)$, the reduction $g_s$ is the map induced by the linear system associated to the reduction of $mL = mf_s ^* H_s - m E_s$. This shows that $-E_s$ is relatively semiample over $X_s$, hence is $f_s$-nef. 
\end{proof}

\begin{lem} \label{codimensiontwolemma}
    Let $f: Y \dashrightarrow Y'$ be a birational map between two normal, $\Q$-factorial projective varieties over $\C$.
    Assume that $f$ is an isomorphism in codimension $1$, i.e., there exist open subsets $U \subset Y$ and $U' \subset Y'$ such that the codimension of $Y \setminus U$ and $Y' \setminus U'$ is at least two in $Y$ and $Y'$ respectively, and such that $f$ restricts to an isomorphism $f|_U : U \to U'$. Then if $Y$ is of Fano type, so is $Y'$.
\end{lem}

\begin{proof}
        Since $f$ is an isomorphism in codimension $1$, $f$ induces an isomorphism of the Cox rings of $Y$ and $Y'$ (by taking strict transforms under $f$). Then, the lemma follows from the main theorem of \cite[Theorem~1.1]{GongyoOkawaSannaiTakagiCoxRings}.
\end{proof}

\begin{rem}
Using the same idea as in \Cref{inductivethm}, one can show that if the conjecture holds in dimension $d$, then it holds for all complexity $d$ log Fano pairs. In particular, if it holds for KLT log Fano pairs $(\bP^1, \Delta)$, then it holds for all complexity $1$ log Fano pairs.
\end{rem}

\section{Three-dimensional singularities} \label{section:threefolds}

In this section, we will focus on three-dimensional singularities and prove some effective comparisons between the lim-inf $F$-signature and the normalized volume.

\begin{defi}[\cite{LiMinimizingNormalizedVolumes}]
\label{defi:normalizedvolume}
Let $x \in (X, \Delta)$ be an $n$-dimensional KLT singularity over $\C$. Then, the \emph{normalized volume} of $(X, \Delta)$ is defined to be
\[ \nvol(X, \Delta) := \inf_{ v \in \DivVal(X,x)} A_{X, \Delta} (v)^n \vol(v)\]
where $A_{X, \Delta}$ denotes the log discrepancy of the valuation and $\vol$ denotes the volume (see \cite{MustataJonssonValuations} for details).
\end{defi}

The limiting value of the $F$-signatures $\s(R_p)$ of the reductions of a complex KLT singularity $R$ as $ p \to \infty$ is expected to exist and be related to the normalized volume of $R$ (\Cref{defi:normalizedvolume}). One precise question relating the limit $F$-signature to the normalized volume is:

\begin{que}
    For any integer $d>0$, does there exist a constant $c := c(d) > 0$ such that for every KLT singularity $(R, \fm)$ over $\C$ of dimension $d$, we have
    \[ \lim_{p \to \infty} \s (R_p) \geq c \nvol(R) ? \]
\end{que}

While we are not able to answer this question completely for three dimensional singularities, the rest of this section is devoted to proving partial results in this direction.

\begin{defi} \label{dfn:alphaboundedKollarcomponent}
    For any real number $\varepsilon \in (0,1)$, we define the set   $\mathscr{S}_\varepsilon$ to be the set of $3$-dimensional local KLT pairs $ (X, \Delta, x)$ such that $\Delta$ has standard coefficients and every Koll\'ar component $(S, \Delta_S) $ of $(X, \Delta, x)$ satisfies $\alpha(S, \Delta_S) \leq 1 - \varepsilon$.
\end{defi} 

\begin{rem}
    We note that if $\varepsilon < 1/2$, then the cone over every log del Pezzo surface $(Y, \Delta_Y)$ where $\Delta_Y$ has standard coefficients and $\varepsilon < \alpha(Y, \Delta_Y) < 1 - \varepsilon $ belongs to the set $\mathscr{S}_\varepsilon.$ Indeed, since the blow-up of the cone point gives a Koll\'ar component isomorphic to $(Y, \Delta_Y)$, it follows from the main theorem of \cite{CheltsovParkShramovAlphaandPLTBlowups} that every other Koll\'ar component $(S, \Delta_S)$ of the cone over $(Y, \Delta_Y)$ satisfies $\alpha(S, \Delta_S) \leq 1 - \alpha(Y, \Delta_Y) < 1 - \varepsilon.$ 
\end{rem}

\begin{thm} \label{thm:effectivelimFsig}
    Fix a real number $  \varepsilon  \in (0,1)$. Then there exists a constant $c := c(\varepsilon)$ such that for every three-dimensional KLT pair $(X, \Delta, x)$ belonging to the set $\mathscr{S}_\varepsilon$, and a family of models $(X_A, \Delta_A, x_A)$ over a smooth finite-type $\Z$-algebra $A$, there exists an open set $U \subset \Spec (A)$ such that for each closed point $s \in U$, we have 
    \[   \s  (X_s, \Delta_s,   x_s) \geq c \nvol (X, \Delta, x).\]
\end{thm}

\begin{cor} \label{cor:logdelPezzo}
    Let $(X, \Delta)$ be a log del Pezzo surface over $\C$, where $\Delta$ has coefficients in the standard set $\{\frac{m-1}{m} \, | \, m \geq 1 \}$. Suppose $\alpha(X, \Delta) < 1$. Then, \Cref{strongerconjversion} holds for $(X, \Delta)$.
\end{cor}

First, we need some preliminary results:

\begin{prop} \label{prop:boundonalphaofKollarcomponent}
    Let $n \geq 2$ be an integer. Then there exists a constant $c := c(n)$ depending only on $n$ such that every $n$-dimensional local KLT pair $(X, \Delta,x)$ admits a Koll\'ar component $(S, \Delta_S)$ such that $\alpha(S, \Delta_S) \geq c$. 
\end{prop}

\begin{rem}
    Colin Fan proved that for any $\epsilon>0$ there exists $S$ with $\delta(S,\Delta_S)\geq 1-\epsilon$. So, \Cref{prop:boundonalphaofKollarcomponent} also follows from this result.
\end{rem}

\begin{proof}
        Let $(X, \Delta, x)$ be an affine, $n$-dimensional KLT pair. To find such a Koll\'ar component, we first find the corresponding valuation $v$ as follows: Let $v_0$ be the minimizer of the normalized volume function of $(X, \Delta, x)$. If $v_0$ is a divisorial valuation, then the corresponding Koll\'ar component $(S, \Delta_S)$ is a K-semistable log Fano variety of dimension $n -1$ by \cite[Theorem A]{LiXuStabilityofvaluations}. Then, we have that $\alpha(S, \Delta_S) \geq \frac{1}{n}$ by \cite[Theorem 3.5]{FujitaOdakadeltainvariant}.
        
        If $v_0$ is not a divisorial valuation, in general $v_0$ is a higher-rank quasi-monomial valuation with a finitely generated associated graded algebra $R_0 = \gr_{v_0}  R$ where $R = \cO_{X,x}$. Then, $X_0 = \Spec(R_0)$ is an affine variety with a good torus action with vertex $x_0$. Moreover,  there are canonically defined $\Delta_0$ and $\xi_0$ such that $\Delta_0$ is an effective $\Q$-divisor such that $(X_0, \Delta_0)$ is a KLT pair and $\xi_0$ is a Reeb vector corresponding to the grading on $R_0$. The triple $(X_0, \Delta_0, \xi_0)$  is a K-semistable log-Fano cone singularity.

        Now, let $(Y, E) \to X$ be a log resolution of $(X, \Delta)$ with reduced exceptional divisor $E$ such that $v_0$ is monomial on $(Y,E)$ with maximal rational rank $r$. This means that there exist exceptional divisors $E_1, \dots, E_r$ and a generic point $\eta$ of a component of $\cap_{i = 1} ^r E_i$ and positive real numbers $\alpha_1, \dots, \alpha_r$ that are linearly independent over $\Q$ such that around $\eta$, $v_0$ is the monomial combination of $E_1, \dots, E_r$ with weights $(\alpha_1, \dots, \alpha_r)$. Now, choose vectors $\underline{\alpha}_i \in \Q^r _{ > 0}$ such that $ \underline{\alpha}_i \to (\alpha_1, \dots, \alpha_r)$ as $i \to \infty$. Let $v_i$ be the monomial combination of $E_1, \dots, E_r$ with weights $\underline{\alpha}_i$. Note that since $\underline{\alpha}_i \in \Q^r _{>0}$, each $v_i$ is a divisorial valuation over $(X,x)$. We will show that any $v_i $ for $i \gg 0 $ will work.

        Firstly, by \cite[Proposition 4.11]{LiXuHigherRank}, each $v_i$ for $i \gg 0$ corresponds to a Koll\'ar component of $(X, \Delta)$. Let $(X_i, \Delta_i, \xi_i)$ be the corresponding degenerations induced by $v_i$. By \cite[Lemma 2.10]{LiXuHigherRank}, we have $\text{gr} _{v_i} R \isom \text{gr}_{v_0} R = R_0$ up to a change in the grading. In other words, for $i \gg 0$, $X_i$ is independent of $i$ and is isomorphic to $X_0$. Moreover, we have the following observation:

        \begin{lem} \label{lem:independenceofboundary}
            For $i \gg 0$, the induced boundary $\Delta_i$ on $X_ 0 \isom X_i$ is independent of $i$ and isomorphic to $\Delta_0$. And identifying $(X_i, \Delta_i, \xi_i) \isom (X_0, \Delta_0, \xi_i)$, where $\xi_i$ is the induced Reeb vector on $X_0$ by the degeneration $v_i$, we have $\xi_i \to \xi_0$ as $i \to \infty$.  
        \end{lem}
       
        Taking this lemma for granted for the moment, consider the function $\Theta(X, \Delta, v_i) = \frac{\nvol (X, \Delta,x)}{\nvol(X_0, \Delta_0; \text{wt}_{\xi_i})}$. By \cite[Lemma 4.10]{LiXuHigherRank}, we have $\nvol(X_0, \Delta_0; \text{wt}_{\xi_i}) = \nvol (X, \Delta; v_i)$. By \cite[Lemma 2.11]{ZhuangBoundednessandmld}, we know that $\Theta (X, \Delta, v_i) \to 1$ as $i \to \infty$ since $v_0$ is the minimizer of the normalized volume of $(X, \Delta)$. Therefore, for $i \gg 0$, we may assume that $\Theta(X, \Delta, v_i) > \frac{1}{2}$.

        Lastly, let $(S_i, \Delta_{S_i})$ denote the Koll\'ar component corresponding to $v_i$. Then, from the discussion in Section 2.4 of \cite{XuZhuangboundednesslogfanocone} (see also \cite[Proposition 2.10]{LiuZhuangSharpnessofTianscriterion}) we have that the pair $(S_i, \Delta_{S_i})$ is isomorphic to the orbifold base $(X_0, \Delta_0)/\xi_i$. Therefore, by applying \cite[Lemma 2.12]{XuZhuangboundednesslogfanocone}, there exists a constant $C:= C(n)>0$ depending only on $n$ such that $ \alpha(S_i, \Delta_{S_i})  \geq \frac{1}{c} \Theta(X, \Delta, v_i) \geq \frac{1}{2C}. $ Now, the proposition follows by taking $ c = \min \{\frac{1}{2C} , \frac{1}{n} \} $.
\end{proof}

\begin{proof}[Proof of \Cref{lem:independenceofboundary}] This proof depends on the geometry of the degenerations induced by $v_i$ as proved in \cite{XuTowardsFiniteGeneration} and \cite{XuZhuangStableDegenerations}. Suppose the minimizer $v_0$ of the normalized volume of $(X, \Delta)$ has rational rank $r$ and is computed at the generic point of a component of $E_1 \cap \dots \cap E_r$ on a log resolution $Y \to X$ of $(X, \Delta)$. Then, there is a flat, $\G_m ^r$-equivariant family $f: (\cX, \Delta_{\cX}) \to \A^r$ such that:
\begin{enumerate}
    \item $(\cX, \Delta_{\cX}) \times_{\A^r} (\A^1 \setminus \{0\})^r \isom (X, \Delta) \times \G_m ^r$,
    \item each fiber $(X_t, \Delta_t)$ of $\pi$ is a KLT pair,
    \item the fiber $(X_0, \Delta_0)$ over $\overline{0} \in \A^r$ of $\pi$ is isomorphic to the degeneration induced by $v_0$, and
    \item for any weight $\underline{w} = (w_1, \dots, w_r) \in \N_{+} ^r$, the pull back of $\pi$ along the map $\A^1 \to \A^r, t \mapsto (t^{w_1}, \dots, t^{w_r})$ is the $\G_m$-equivariant degeneration induced by $v_{\underline{w}}$, the quasi-monomial combination of $E_1, \dots, E_r$ with weights $\underline{w}$.
\end{enumerate}  
See \cite[Proposition 4.5]{XuZhuangStableDegenerations} for Parts (1) and (2) and \cite[Corollary 4.10]{XuZhuangStableDegenerations} for Parts (3) and (4). For Part (3), note that by the choice of $E_1, \dots, E_r$, the minimizer $v_0$ corresponds to weights in the interior of $\R_{\geq 0} ^r$. Therefore, for all weights $\underline{w} \in \N_+ ^r$, the degeneration induced by the corresponding quasi-monomial valuation $v_{\underline{w}}$ has a central fiber $(X_0, \Delta_0)$ isomorphic to the fiber over $\overline{0} \in \A^r$ of $\pi$. Since $v_i$ were chosen to be monomial combinations of $E_1, \dots, E_r$ with positive weights, this proves the first part of the lemma.

Note that while the central fiber of the degeneration induced by $v_{\underline{w}}$ is the same for all weights $\underline{w} \in \N_+ ^r$, the $\G_m$ action on $(X_0, \Delta_0)$ depends on the weight $\underline{w}$. Let $\xi_{\underline{w}}$ denote the Reeb vector induced by $v_{\underline{w}}$. By \cite[Lemma 3.8]{XuTowardsFiniteGeneration}, the Reeb vector $\xi_{\underline{w}}$ is equal to $\underline{w}$, and therefore, since the weights $\underline{\alpha}_i \to (\alpha_1, \dots, \alpha_r)$ (where $(\alpha_1, \dots, \alpha_r)$ are the weights for the minimizer $v_0$, the Reeb vectors $\xi_{i} \to \xi_0$ as $i \to \infty$. 
\end{proof}

Next, we consider the case of cone singularities.

\begin{prop} \label{prop:effectiveboundforFanosurfaces}
    Let $\varepsilon \in (0, 1/2)$ be a fixed real number. Then, there exists a constant $C := C(\varepsilon) > 0$ such that for every two dimensional log Fano pair $(S, \Delta_S)$ such that $ \varepsilon < \alpha(S, \Delta_S) < 1 - \varepsilon $ and where $\Delta_S$ has standard coefficients, we have 
    \[  \s_{-} (S, \Delta_S) \geq C \vol (S, \Delta_S).\]
\end{prop}
\begin{proof}
    Given $\varepsilon \in (0, 1/2)$, first we claim that there is a constant $C'>0$ such that for all two dimensional log Fano pairs $(S, \Delta_S)$ (where $\Delta_S$ has standard coefficients) and $\varepsilon < \alpha(S, \Delta_S) < 1 - \varepsilon $, we have $\alpha_- (S, \Delta_S) \geq C$. Given this claim, the proposition follows from the inequality
    \[ \s_- (S, \Delta_S) \geq \frac{1}{24} \alpha_- (S, \Delta_S) ^3  \vol( - K_S - \Delta_S). \]

    To show the above claim, we note that the proof of \Cref{inductivethm} shows that for any pair $(S, \Delta_S)$ as above, $\alpha_{-} (S, \Delta_S) \geq c \alpha_{-} (E', \Delta_{E'})$ (see \Cref{equation:mainclaim}) where $(E', \Delta_{E'}) $ is an appropriate one-dimensional log Fano pair. Moreover, by \Cref{lem:alphaofE'} and \Cref{eqn:alpha0bound}, we know that $\alpha(E', \Delta_{E'}) \geq \frac{\varepsilon}{2} $. Therefore, it suffices to bound $\alpha_- (E', \Delta_{E'}) $  from below under these hypotheses. This follows from the results in \cite{BCPTlimitFsigarxiv} as discussed below (in particular, from \Cref{eqn:limalphaP1pairs}).
\end{proof}

\begin{rem} [One-dimensional log Fano pairs with standard coefficients] 
    A one dimensional log Fano pair with standard coefficients is of the form $(E, \Delta_E) = (\P^1 , \delta_1 p_1 + \delta_2 p_2 + \delta_3 p_3)$ where $p_1,p_2,p_3$ are distinct points on $\P^1$ (which can be assumed to be the points $0, 1 , \infty$), and $\delta_1, \delta_2, \delta_3 \in \{0 \} \cup \{ \frac{m-1}{m} \, | \, m \geq 1 \}$ and such that $\delta_1 + \delta_2 + \delta_3 < 2.$ In this notation, assuming $\delta_1 \geq \delta_2 \geq \delta_3 $, we have
    \begin{equation} \label{eqn:alphaindimone} \alpha (E, \Delta_E) = \frac{1 - \delta_1}{2 - (\delta_1 + \delta_2 + \delta_3) }  .\end{equation}
    Using this formula, we see that $\alpha (E, \Delta_E) \geq 1/2$  unless $ \delta_1 > \delta_2 + \delta_3 $. In the former case, the pair $(E, \Delta_E)$ is K-semistable and by the main theorem of \cite{BCPTlimitFsigarxiv} it follows that 
    \[ \lim_{p \to \infty} \s_p (E, \Delta_E) = (1 - \frac{\delta_1 + \delta_2 + \delta_3}{2}) ^2  = \frac{\vol( -K_E - \Delta_E)}{4}.  \]
    Then, it follows from \cite[Corollary~5.8]{PandeFrobeniusAlpha} that $\lim_{p \to \infty} \FA(E, \Delta_E) = 1/2,$ so that $\alpha_{-} (E, \Delta_E) = 1/2$.

    On the other hand, when $\alpha(E, \Delta_E) < 1/2$, by the example computed in \cite[Section 4.1]{LiStabilityofSheaves}
    the pair $(E, \Delta_E)$ admits an isotrivial degeneration to the pair $(E_0, \Delta_{E_0}):= (\P^1, \delta_1 p + (\delta_2 + \delta_3) q)$.
    Thus, for every reduction to characteristic $p \gg 0$, we have $\FA(E_p, \Delta_{E_p}) \geq \FA(E_0. \Delta_0)$ by the semicontinuity of $\FA$ in families. Furthermore, since $(E_0, \Delta_{E_0})$ is isomorphic to a toric pair (after a linear automorphism), we know that $\FA(E_0, \Delta_{E_0}) = \alpha_F (E_0, \Delta_{E_0})$ for a reduction of $(E_0, \Delta_{E_0})$ to any characteristic $p>0$. This follows from a similar argument as \cite[Theorem~5.15]{PandeFrobeniusAlpha}\footnote{The Gr\"obner degeneration induced by any monomial order preserves the boundary $\Delta_E$ and hence $\FA(E, \Delta_E)$ is computed by monomials, in which case $F$-regularity coincides with being KLT.}. Lastly, by inspecting \Cref{eqn:alphaindimone} we see that $\alpha(E, \Delta_E) = \alpha(E_0, \Delta_{E_0})$. Therefore, we have
    \[\lim_{p \to \infty} \FA(E, \Delta_E) \geq \alpha_F (E_0, \Delta_{E_0}) = \alpha(E_0,\Delta_{E_0}) = \alpha(E, \Delta_E).    \]
    Since we always have $\alpha_F (E, \Delta_E) \leq \alpha(E, \Delta_E)$, we obtain that $\lim_{p \to \infty} \FA (E, \Delta_E) = \alpha (E, \Delta_E)$.

    In summary, in either case we always have
    \begin{equation} \label{eqn:limalphaP1pairs} \lim_{p \to \infty} \FA(E, \Delta_E) =  \min \{ \frac{1}{2}, \alpha(E, \Delta_E) \} . \end{equation}
\end{rem} 

\begin{proof}[Proof of \Cref{thm:effectivelimFsig}]
    Fix $\varepsilon \in (0,1)$ and the constant $C := C(2) $ as in \Cref{prop:boundonalphaofKollarcomponent}. Given a singularity $(X, \Delta, x)$ in $\mathscr{S}_{\varepsilon}$, we first use \Cref{prop:boundonalphaofKollarcomponent} to pick a Koll\'ar component $(S, \Delta_S)$ of $(X, \Delta, x)$ such that $\alpha(S, \Delta_S) \geq C$. By the definition of $\mathscr{S}_\varepsilon$, we also know that $\alpha(S, \Delta_S) < 1 - \varepsilon.$ Applying \Cref{prop:effectiveboundforFanosurfaces}, we see that there is a constant $C' >0$ depending only on $\varepsilon $ and an open set $U \subset \Spec (A)$ such that for each closed point $s \in U$, we have $\mathscr{s}(S_s, \Delta_{S_s}) \geq C' \vol (S, \Delta_S)$. Let $A := A_{(X, \Delta)} (S) $. By \Cref{Prop:lowerboundKollarcomponent}, we know that $\s (X_s, \Delta_s, x_s) \geq A \, \s (S_s, \Delta_{S_s}) = A^3 \, \vol(\ord_S) = \nvol(\ord_S)$. Therefore, we conclude that $\s (X_s, \Delta_s, x_s) \geq C' \nvol(S, \Delta_S) \geq C' \nvol (X, \Delta, x)$.  
\end{proof}

\bibliographystyle{alpha}
\bibliography{ref}

\end{document}